\documentclass[12pt,reqno]{article}
 \usepackage[body={6.5in,9.0in},left=1in,top=1in,centering]{geometry}
 
 \usepackage{amssymb,amsmath,amsthm, amsfonts,color,pstricks,mathrsfs,mathtools}
\usepackage[sort,comma]{natbib}
\usepackage[title]{appendix}

\newtheorem{thm}{Theorem}[section]
\newtheorem{cor}[thm]{Corollary}

\newtheorem{prop}[thm]{Proposition}
\newtheorem{lem}[thm]{Lemma}

\newtheorem{defn}[thm]{Definition}
\newtheorem{cnd}[thm]{Condition}

\newtheorem{rem}[thm]{Remark}

\newcommand{\thmref}[1]{Theorem~{\rm \ref{#1}}}
\newcommand{\lemref}[1]{Lemma~{\rm \ref{#1}}}
\newcommand{\corref}[1]{Corollary~{\rm \ref{#1}}}
\newcommand{\cndref}[1]{Condition~{\rm \ref{#1}}}
\newcommand{\propref}[1]{Proposition~{\rm \ref{#1}}}
\newcommand{\defref}[1]{Definition~{\rm \ref{#1}}}
\newcommand{\remref}[1]{Remark~{\rm \ref{#1}}}

\newcommand{\sectref}[1]{Section~{\rm \ref{#1}}}
\newcommand{\comment}[1]{}

\makeatletter \@addtoreset{equation}{section}

\allowdisplaybreaks

 \def\R{\ensuremath {\mathbb R}}
\newcommand{\EE}{\mathbb E}
\newcommand{\PP}{\mathbb P}
\newcommand{\QQ}{{\cal Q}}
\newcommand{\Q}{{\mathbb Q}}
\newcommand{\NN}{\mathbb N}
\newcommand{\RR}{\mathbb R}
\def\P{\ensuremath{\mathbb P}}

\newcommand{\F}{\mathcal F}\def\G{\mathcal G}
\newcommand{\FFF}{{\mathscr F}}

\newcommand{\B}{{\mathcal B}}
\newcommand{\E}{{\mathcal E}}
\newcommand{\Z}{{\mathcal Z}}

\renewcommand{\P}{{\mathcal P}}

\newcommand{\wdh}{\widehat}
\newcommand{\wdt}{\widetilde}
\newcommand{\ovl}{\overline}
\newcommand{\cadlag}{c\`{a}dl\`{a}g }

\newcommand{\ladlag}{l\`{a}dl\`{a}g }

\newcommand{\htau}{\wdh\tau}

\newcommand{\ttau}{\wdt\tau}
\newcommand{\mbf}{\mathbf}

\title{On the Modelling of  Impulse Control with Random Effects for Continuous Markov Processes
}

\author{K.L. Helmes\thanks{ Institute for Operations Research, Humboldt University of Berlin, Spandauer Str. 1, 10178, Berlin, Germany, {\tt helmes@wiwi.hu-berlin.de}.}\and  R.H. Stockbridge\thanks{Department of Mathematical Sciences,   University of Wisconsin-Milwaukee,   Milwaukee, WI 53201,   USA,   {\tt stockbri@uwm.edu}; {\tt zhu@uwm.edu}.} \and C. Zhu\footnotemark[3]}

\begin{document}
\maketitle

\begin{abstract}
The use of coordinate processes for the modelling of impulse control for {\em general}\/ Markov processes typically involves the construction of a probability measure on a countable product of copies of the path space.  In addition, admissibility of an impulse control policy requires that the random times of the interventions be stopping times with respect to different filtrations arising from the different component coordinate processes.  When the underlying Markov process has {\em continuous}\/ paths, however, a simpler model can be developed which takes the single path space as its probability space and uses the natural filtration with respect to which the intervention times must be stopping times.  Moreover, this model construction allows for impulse control with random effects whereby the decision maker selects a distribution of the new state.  This paper gives the construction of the probability measure on the path space for an admissible intervention policy subject to a randomized impulse mechanism.  In addition, a class of polices is defined for which the paths between interventions are independent and a further subclass for which the cycles following the initial cycle are identically distributed.  A benefit of this smaller subclass of policies is that one is allowed to use classical renewal arguments to analyze long-term average control problems.    Further, the paper defines a class of {\em stationary}\/ impulse policies for which the family of models gives a Markov family.  The decision to use an $(s,S)$ ordering policy in inventory management provides an example of an impulse policy for which the process   has i.i.d.~cycles and the family of models forms a Markov family.
  \vspace{2 mm}
	
\vspace{1 mm}

\end{abstract}

\section{Introduction}
Impulse control was introduced by \cite{bens:75} and arises naturally in a wide variety of applications such as the management of inventory, exchange rates, and financial portfolios.  It occurs when the state of the system is adjusted in a discontinuous fashion and the cost includes a fixed positive charge for each such intervention. 

This paper examines the construction of the mathematical model for impulse control.  Intuitively, the state process evolves as a Markov process until the decision maker intervenes to instantly move the state to a new location (at a cost) after which the process again evolves as the Markov process starting from this new state.  The decision maker then waits to intervene again to instantly move the process at an additional cost and these actions continue into the infinite future.  The model as described can be clearly understood so some papers simply assume the existence of the impulse-controlled state process (see e.g., \cite{rich:77}, 
\cite{cade:00}, \cite{rung:18} and many others).  

For some applications one desires additional structure to the model such as independence of the evolutions of the process between interventions for some classes of intervention policies.  Intuitively it is again ``obvious'' that such a model is possible but it is challenging to move beyond the informal description of the process to a specific mathematical model; something that has been described as a ``hard problem'' (see \cite{mena:17}) for which the ``formal probabilistic apparatus \ldots is unfortunately rather cumbersome'' \cite[p.~227]{davi:93}.  

This paper constructs the mathematical model for impulse control of a process on the space of (augmented) \cadlag functions when there is uncertainty about the result of the intervention.  Due to these random effects, the filtration with respect to which the decisions are made is generated by the controlled process.  Beyond the construction of the model, three subclasses of policies are identified such that   (i) the controlled process between interventions form independent cycles, (ii) these independent cycles are identically distributed after the first intervention, and (iii) the family of models forms a Markov family.

A typical approach to defining the impulse-controlled process is to start by setting the sample space $\Omega$ to be the path space.  For example, \cite{harr:83} and \cite{orme:08} define the uncontrolled process to be a drifted Brownian motion in $\RR$ while \cite{bens:84} examines a more general strong solution to a stochastic differential equation in $\RR^n$.  Since the fundamental evolution of the process for both of these models is continuous, these papers set $\Omega=C_{\RR^n}[0,\infty)$ with $n=1$ in the first two papers.  The authors let $X$ denote the coordinate process, ${\cal F}=\sigma(X(t): t\geq 0)$ and use the natural filtration $\{{\cal F}_t\}$ in which ${\cal F}_t = \sigma(X(s): 0 \leq s \leq t)$.  The impulse policies consist of an increasing sequence of intervention times, say $\{\tau_k\}$, and a sequence of impulse random variables $\{\xi_k\}$.  Each intervention time $\tau_k$ must be an $\{{\cal F}_t\}$-stopping time and the corresponding impulse $\xi_k$ must be ${\cal F}_{\tau_k}$-measurable.  

This path space approach is also used by \cite{robi:78}, \cite{stet:83} and \cite{lepe:84} but with more complexity in that $\Omega=D_{\mathcal E}[0,\infty)$ so the coordinate process $X$ includes the possibility of the process dynamics having jumps.  The model, however, is built on the countable product $\wdt\Omega = \prod_{i=0}^\infty \Omega$ in which the different components are used for the evolution of the state process following the different interventions.  In each of these models, the intervention decisions are made subject to different filtrations $\{{\cal F}^n_t\}$, $n\in \NN_0:=\{0\}\cup\NN$, in which $\{{\cal F}^n_t\}$ is generated by the coordinate processes in the first $n$ components.  A small but significant difference between these constructions occurs at the times when $X$ jumps and the decision maker immediately intervenes; such a time is one of the intervention times $\tau_k$.  The models in  \cite{robi:78} and \cite{stet:83} define $X$ over the successive half-closed intervals $\{[\tau_k,\tau_{k+1}): k \in \NN_0\}$ resulting in \cadlag paths.  However, the actual value to which $X$ moves prior to the intervention which causes this intervention is never captured by observing the state process.  A difference between the models in \cite{robi:78} and \cite{stet:83} is that the latter paper explicitly considers the possibility of multiple interventions at the same time.  \cite{chri:14} adopts the construction of Stettner and tries to distinguish between the three locations $X_{\tau_n-}$, $X_{\tau_n,-}$ and $X_{\tau_n}$ where the first is the left limit of $X$ at the time $\tau_n$, the second gives the location of $X$ following the natural jump of the process, if any, at time $\tau_n$ but before the impulse is applied and the last expression gives the state of $X$ after the impulse takes effect.  Unfortunately, the model of \cite{stet:83} defines $X$ over the half-closed intervals $\{[\tau_k,\tau_{k+1}): k \in \NN_0\}$ so the position $X_{\tau_n,-}$ is never part of the information in the natural filtration.  

\cite{lepe:84} adjusts the model by having the impulse only take effect {\em after}\/ the time of intervention so the natural filtration of this state process observes $X$ after the jump but before the impulse.  The impulse control model in \cite{davi:93} is quite similar to that of \cite{lepe:84} but applies the construction to piecewise deterministic processes.  One unfortunate aspect of these latter, more precise models is that the paths of the impulse-controlled process will not be \cadlag at jump times of the process which immediately bring about interventions.  \cite{yush:83,yush:89} also constructs models for impulse control of piecewise deterministic processes which result in \cadlag processes.

Preceding the introduction of impulse control models, \cite{iked:66} used a similar construction to extend a killed Markov process to infinite time so that at each time of death the process is reinitialized and in doing so, the authors define a {\em single}\/ filtration with respect to which the rebirth times are stopping times and the resulting process is Markov.  This addresses the issue of having multiple filtrations of information, though the single filtration is quite complex.  Complete proofs of the construction are given in \cite{meye:75}.

An alternate approach for modelling the impulse-controlled process (see e.g., \cite{mena:80}, \cite{alva:04}, \cite{okse:05}, 
\cite{jack:06}, \cite{frey:10}, \cite{helm:15}) is to start with a given filtered probability space $(\Omega,{\cal F},\{{\cal F}_t\},\PP)$ on which the fundamental evolution of the process can be defined for each initial distribution.  The impulse-controlled process is then constructed iteratively over the successive intervals $\{[\tau_k,\tau_{k+1})\}$ by pasting together a shift of the fundamental process having the required different initial positions given by the impulses.  The impulse policy requires each intervention time to be an $\{{\cal F}_t\}$-stopping time.  Since $\{{\cal F}_t\}$ is assumed given, it is possible that more information than that generated by the state process is included in this filtration.

Another common approach to the modelling of impulse-controlled process is to simply refer to the constructions given in one of the aforementioned papers (see e.g., \cite{korn:97}, \cite{mena:17} and \cite{palc:17}).


This paper develops a simple model for the impulse control of Markov processes having {\em continuous} paths which provides extra properties of the process.  The path continuity implies that the left limit $X(\tau_k-)$ is always the state at which the impulse is applied which therefore has two important consequences.  First, when the impulse occurs {\em at} time $\tau_k$, the natural filtration includes the ``jump from state'' and, second, the resulting path is c\`{a}dl\`{a}g.  Both of these observations contrast with the models of \cite{lepe:84} and \cite{davi:93} when the underlying process has inherent discontinuities in its paths.  Their characteristic that the impulse takes effect immediately following $\tau_k$ is needed for the state from which the process jumps to be included in the natural filtration but, in general, leads to \ladlag paths.   

The simplicity of our model is that it is built using the coordinate process $X$ on the filtered space $(\Omega,{\cal F},\{{\cal F}_t\})$ in which $\Omega=D_\E[0,\infty)$ with $\E$ being the state space, ${\cal F}=\sigma(X(s): s \geq 0)$ and the filtration is the natural filtration generated by $X$ so ${\cal F}_t = \sigma(X(s): 0 \leq s \leq t\}$.  It is shown that, for each admissible impulse control policy $(\tau,Z)$ defined below and probability measure $\nu$ on $(\E,\B(\E))$, there exists a probability measure $\PP_{\nu}^{(\tau,Z)}$ on a simple augmentation of $(\Omega,{\cal F})$ under which $X$ is the desired impulse-controlled process.  The augmentation of $\Omega$ is to account for the possibility of an initial intervention at time $0$; see \sectref{math-model} for details of the augmentation.

In contrast with almost all papers on impulse control previously mentioned, a special feature of the model is the inclusion of uncertainty in the application of the impulses.  By this we mean that we select a control variable which, in conjunction with the state of the process at the instant before the intervention, determines a {\em distribution}\/ on the state space to which the impulse moves the process.  An example of such a random effect impulse arises in inventory management where an order is placed but only some random fraction of the ordered amount is delivered or perhaps the entire order is delivered but some random quantity must be discarded due to manufacturing defects.  A very similar model is studied in \cite{korn:97} for the case of a one-dimensional diffusion given as the solution to a stochastic differential equation.  Informally defining the model for the impulse-controlled state process, Korn refers to the construction of \cite{bens:84} and indicates that this can be adapted to allow a distribution for the post-impulse location.  Moreover, he assumes that the fundamental diffusion $X$ is defined on a given filtered probability space.  As indicated above, our model lives on the augmented $\Omega$ using the natural filtration of the coordinate process.

Motivated by consideration of a long-term average cost criterion, a second property of our model is that it results in the independence of the controlled process over the different intervals between interventions when the policy is from a subclass of admissible policies.  Further, all but possibly the first of these cycles will be identically distributed when the policy is in a further subclass.  In order to have this independence, we use a countable product space similar to \cite{robi:78}, \cite{stet:83}, \cite{bens:84} and \cite{lepe:84} but the impulse policy is initially defined on the augmented $(\Omega,{\cal F},\{{\cal F}_t\})$ and then carefully related to the product space.  In this manner a probability measure is established on the countable product space first and  the desired impulse-controlled process $\wdt{X}$ is defined similarly as in the other references.  However, since $\wdt{X}$ is c\`adl\`ag, its distribution on the augmented $(\Omega,{\cal F})$ then provides the desired measure $\PP_x^{(\tau,Z)}$.


The focus of all the impulse control papers we reviewed is on describing the controlled process and then analyzing problems of interest.  Those papers which detail the construction provide a model for the process, based on having a Markov or strong Markov process for the evolution between interventions, but they do not address whether the resulting controlled process is Markov or strong Markov.  It is somewhat obvious that a general result is not possible for {\em all}\/ policies since the admissible policies only require the intervention times to be stopping times and the intervention amounts to be measurable with respect to the associated stopped $\sigma$-algebras.  For example, the decision maker could define a policy that only depends on the initial position of the process.  The resulting controlled process cannot be Markov since the information about the initial position is not included in the $\sigma$-algebra generated by the state of the process at or just prior to an intervention time.    We identify a subclass of ``stationary'' policies for which the coordinate process, its natural filtration and the resulting family of measures forms a Markov family.  As pointed out to the authors by one of the referees, having a Markov family facilitates, for instance,  the verification of a transversality condition for infinite-horizon or ergodic functionals of the resulting time-homogeneous controlled processes.

The paper is organized as follows.  Section \ref{math-model} describes the fundamentals of the underlying Markov process, the distributions determined by the interventions which select the new states to which the process moves and defines the class of admissible nominal impulse policies.  Given such an admissible policy, the existence and uniqueness of the corresponding measure on the countable product of spaces is proven in Section \ref{sect-existence} and this is used as indicated above to obtain the measure on the space of trajectories.  Sufficient conditions on a policy for independence and identically distributed cycles are briefly presented in Section \ref{sect-ind-cycles}   while the subclass of stationary Markov policies resulting in a Markov family  is covered in  Section \ref{sect-Markov}.  The paper concludes with Section \ref{sect-example} by showing that the $(s,S)$ ordering policy of inventory management results in the controlled process being Markov and having iid cycles.

\section{Model Fundamentals} \label{math-model}

{\bf \em Process Dynamics.}\/ The model consists of dynamics which describe the evolution of the process in the absence of any interventions (as well as between the interventions).  Let $\E$ be a complete, separable metric space in which the process evolves.  Since the impulse-controlled process will have at most countably many discontinuities arising from the intervention decisions, we choose to describe the model using the space $\Omega:=D_\E[0,\infty)$ of \cadlag functions.  Let $X: \Omega\rightarrow D_\E[0,\infty)$ be the coordinate process such that $X(t,\omega) = \omega(t)$ for all $t\geq 0$, let ${\cal F} = \sigma(X(t): t\geq 0)$ and $\{{\cal F}_t\}$ be the natural filtration.   Furthermore let $\{\PP_x: x\in {\cal E}\}$ be a family of measures on $(\Omega,{\cal F})$ such that $(\Omega,{\cal F}, X,\{\F_{t}\},\{ \PP_{x},x\in \mathcal E\})$ is a Markov family in the sense of Definition 2.5.11 of \cite{kara:88}.  

We  assume throughout the paper that the family $\{\PP_x: x\in \E\}$ satisfies the following support condition:

\begin{cnd} \label{Markov-family}
For each $x\in \mathcal E$,   $\PP_x$ has its support in $C_\E[0,\infty)\subset \Omega$.  In other words, the coordinate process $X$ is continuous $\PP_x$-almost surely.  We call $X$ the fundamental Markov process.
\end{cnd}

\begin{rem}
 The construction of the model in \thmref{prob-construction} only uses the universal measurability condition of the Markov family requirement on the fundamental process $X$ to establish \thmref{prob-construction}(b).  The results in \thmref{prob-construction}(a) therefore apply to more general fundamental processes.  The full Markov family property is required, however, in our proofs of Theorems  \ref{Markov-family-thm} and   \ref{lem-family-equivalence}.  We choose to uniformly adopt the Markov family property of the fundamental process throughout the manuscript. 
\end{rem}

\noindent
{\bf \em Random Effect Impulse Mechanism.}  Let $({\cal Z},\mathfrak{Z})$ be a measurable space representing the impulse control decisions.
Let $\QQ=\{Q_{(y,z)}: (y,z) \in {\cal E}\times {\cal Z}\}$ be a given family of probability measures on ${\cal E}$ such that for each $\Gamma\in {\cal B}({\cal E})$, the mapping $(y,z) \mapsto Q_{(y,z)}(\Gamma)$ is ${\cal B}({\cal E})\otimes \mathfrak{Z}$-measurable.  

\noindent
{\bf \em Random Effect Impulse Mechanism: Special Case.} 
Let $\QQ=\{Q_{(y,z)}: (y,z) \in {\cal E}^2\}$ be a given family of probability measures on ${\cal E}$ such that for each $\Gamma\in {\cal B}({\cal E})$, the mapping $(y,z) \mapsto Q_{(y,z)}(\Gamma)$ is measurable with respect to ${\cal B}({\cal E}^2)$.  

The way to view the random effect impulses in the special case is that at a time when the decision maker intervenes, the controlled process is at $y$ and the aim is to instantly move the process to $z$ while the distribution of the actual position of the process following this intervention is given by $Q_{(y,z)}$.  We refer to $z$ as the nominal impulse since this is the place to which the process aims to jump by this intervention.  Motivated by the special case, we also use the moniker of ``nominal impulse'' to refer to the choice of impulse control $z \in {\cal Z}$ in general since it determines the distribution $Q_{(y,z)}$ of the new location of the process following the intervention.

We now define a nominal impulse policy.  In order to do so, we need to specify the filtration of information used by the decision maker to determine the nominal interventions.  Let $\{{\cal F}_{t-}\}$ be given by ${\cal F}_{t-} = \sigma(X(s): 0 \leq s < t)$ for $t > 0$ with $\F_{0-}=\{\emptyset,\Omega\}$.  It is also important to specify the $\sigma$-algebras of information available at and prior to a stopping time.  Let $\eta$ be an $\{\F_{t-}\}$-stopping time.  Then 
${\cal F}_{\eta} := \{A\in \F: A\cap \{\eta\leq t\}\in {\cal F}_{t-}, t\geq 0\}$ while 
${\cal F}_{\eta-} := \sigma(\{A\cap \{\eta > t\}:  A\in {\cal F}_{t-}, t\geq 0\})$. 

A little more care in the definition of the filtration is necessary for the first intervention.  Observe that every $\omega\in D_\E[0,\infty)$ is right continuous at $0$ which precludes the possibility of an intervention occurring at time $0$.  To allow the possibility of such an intervention which then moves $X$ to a new location $X(0)$, we need to {\em augment}\/ the space $D_\E[0,\infty)$ so that it contains the location from which the intervention occurs.  Define $\check{\Omega} = \E \times D_\E[0,\infty)$, set $\check\F=\B(\E)\otimes \F$ and denote elements $\check\omega\in \check\Omega$ by $\check\omega = (\check\omega(0-),\check\omega(\cdot))$.  Extend the coordinate process $X$ on $D_\E[0,\infty)$ to $\check\Omega$ by defining $X(0-,\check\omega) = \check\omega(0-)$ while keeping $X(s)=\check\omega(s)$ for $s\geq 0$.  Similarly extend the filtrations $\{\F_t\}$ and $\{\F_{t-}\}$ to $\{\check{\cal F}_t\}$ and $\{\check{\cal F}_{t-}\}$ by setting $\check{\cal F}_t = \B(\E)\otimes \F_t$ and $\check{\cal F}_{t-} = \B(\E)\otimes \F_{t-}$, respectively, for $t \geq 0$.  For each $x\in \E$, extend the measure $\PP_x$ on $(\Omega, \F)$ to a measure $\check\PP_x$ on $(\check\Omega, \check\F)$ by putting full mass on the subset $\{\check\omega \in \check\Omega: \check\omega(0-)=x\}$.  Finally, define the $\sigma$-algebra $\check\F_{\eta-} = \B(\E)\otimes \F_{\eta-}$.

We claim that $(\check\Omega,\check\F,X,\{\check\F_t\},\{\check\PP_x:x\in \E\})$ inherits the Markov family properties.  To see this, notice that for each $x\in \E$, $\check\PP_x\{\check\omega\in \check\Omega: \check\omega(0-)\neq \check\omega(0)\}=0$ by definition.  Intuitively, think of $X(0-)$ on $\check\Omega$ as being a ``left limit of $X$ at 0.''  Then the extension of $(\Omega,\F,\PP_x)$ to $(\check\Omega,\check\F,\check\PP_x)$ is defined so that $X$ is continuous at $0$, $\PP_x$-almost surely.  It then naturally  follows that $(\check\Omega,\check\F,X,\{\check\F_t\},\{\check\PP_x:x\in \E\})$ is a Markov family.  This observation will be used to establish a universal measurability property in the proof of \thmref{prob-construction} as well as  the Markov family property in  Theorems  \ref{Markov-family-thm} and \ref{lem-family-equivalence}.

\begin{defn}[Nominal Impulse Policy] \label{nominal-inventory-policy1}
A nominal impulse policy $(\tau,Z) \newline= \{(\tau_k,Z_k): k \in \NN\}$ is a sequence of pairs defined on $(\check\Omega,\check\F)$ in which: (a) the sequence $\{\tau_k\}$ satisfies (a,i) $\tau_1$ is an $\{\check\F_{t-}\}$-stopping time and for $k\geq 2$, $\tau_k$ is an $\{{\cal F}_{t-}\}$-stopping time; (a,ii) for each $k\in \NN$, on the set $\{\tau_k < \infty\}$, $\tau_{k+1} > \tau_k$; (a,iii) $\lim_{k\rightarrow \infty} \tau_k = \infty$; and (b) for each $k \in \NN$, $Z_k$ is a ${\cal Z}$-valued, ${\cal F}_{\tau_k-}/\mathfrak{Z}$-measurable random variable ($Z_1$ being $\check\F_{\tau_1-}/\mathfrak{Z}$-measurable).
\end{defn}

The subtle requirement that $\tau_1$ be an $\{\check\F_{t-}\}$-stopping time and for $k \geq 2$, each $\tau_k$ be a stopping time relative to $\{{\cal F}_{t-}\}$ prevents the decision maker from making a decision based on seeing the new location of $X$ at an intervention time.  
In addition we observe that ultimately each $Z_k$ will only be relevant on the set $\{\tau_k < \infty\}$ since no $k$th impulse action is taken on the set $\{\tau_k = \infty\}$.

An important distinction between our model and those previously mentioned in our literature review is that the nominal impulse policy is based on the filtration generated by the controlled process rather than the filtration generated by the fundamental process.  This use of the controlled process's augmented natural filtration is required due to the random effect of the impulses since later decisions need to be based on the realizations of the earlier impulses.  

 Our nominal impulse policy $(\tau,Z)$  does not allow  two or more simultaneous impulses.  This requirement imposes a restriction on the interplay between the random effect distribution $Q_{(y,z)}$ of one intervention and the activation of the next intervention.  Specifically, the support of $Q_{(y,z)}$ must be disjoint from the set of locations at which the next impulse is initiated. 

For many applications, the decision criterion penalizes simultaneous impulses through the imposition of a fixed positive cost for each distinct intervention so an optimal policy will separate the impulse times, provided the random effects distributions and interventions are well-behaved as in the paragraph above.  Thus, the requirement  that  the intervention times are strictly increasing is natural and not really a restriction in many applications.  
 
\begin{rem}
It is possible to introduce relaxed impulse controls in which the choice of nominal impulse at an intervention is itself random.  Thus at an intervention, a distribution $\pi$ on $Z$ is used, resulting in a random effects distribution 
$$Q_{(y,\pi)}(\cdot) = \int Q_{(y,z)}(\cdot)\, \pi(dz).$$  
Naturally, the sequence $\boldsymbol{\pi} = \{\pi_k\}$ needs to satisfy the measurability condition in \cndref{nominal-inventory-policy1}(b).  We choose to construct the model without this additional level of randomization.
\end{rem}

\section{Existence Result}\label{sect-existence}
Before constructing the model for the impulse-controlled process, it is helpful to state a set of necessary and sufficient conditions on a random time $\tau_k: \Omega \rightarrow [0,\infty]$ for it to be an $\{{\cal F}_{t-}\}$-stopping time.  These conditions play a subtle but important role in our construction of the measure on the countable product space of \cadlag paths.  A similar result was first obtained in \cite{galm:63} when $\tau_k$ is an optional time; that is, when $\{\tau_k < t\}$ is ${\cal F}_t$-measurable for each $t \geq 0$.  The analogous characterization of a random time $\tau_k$ being an $\{{\cal F}_t\}$-stopping time was obtained in Theorem~1.3 of \cite{cour:65}.  It also holds when $\tau_k$ is an $\{{\cal F}_{t-}\}$-stopping time. We state the result in the form we use; recall $\F_{0-}=\{\emptyset,\Omega\}$.  The theorem relies on the following equivalence relation between paths $\omega_1,\omega_2 \in \Omega$: for each $t \geq 0$, $\omega_2 \stackrel{R_t}{\sim} \omega_1$ holds if and only if $\omega_2(s) = \omega_1(s)$ for all $s < t$.

\begin{thm} \label{cp65}
For a mapping $\tau: \Omega \rightarrow [0,\infty]$ to be a stopping time with respect to $\{{\cal F}_{t-}\}$, it is necessary and sufficient that  (i) $\tau$ be ${\cal F}$-measurable and (ii) that for all $t\geq 0$, 
if $\omega_2\stackrel{R_t}{\sim}\omega_1$ and $\tau(\omega_1) \leq t $, then $ \tau(\omega_2)=\tau(\omega_1)$. 
\end{thm}

\begin{proof}
Theorem~1.3 of \cite{cour:65} has $\tau$ being an $\{{\cal F}_t\}$-stopping time with $R_t$ defined using $s\leq t$; its proof relies on the result in their Lemma~1.2.  These results remain valid using $\{\F_{t-}\}$ and $R_t$ defined above, with only minor changes to the proofs.  For example, their stopping operator $\alpha_t  \omega(s) := \omega_{s\wedge t}$ must be replaced by the left-limit $\alpha_t \omega(s) := \omega_{(s\wedge t)-}$. 
\end{proof}

\begin{rem} \label{cp65-tau1-mod}
Let $(\tau,Z)$ be a nominal impulse policy.  \thmref{cp65} holds for each $\{\F_{t-}\}$-stopping time $\tau_k$ for $k \geq 2$ since $\{\tau_k \leq 0\} = \emptyset \in \F_{0-}$.  The theorem can be modified for the stopping time $\tau_1$ by using $\{\check\F_{t-}\}$ and the relation between paths $\check\omega_1,\check\omega_2 \in \check\Omega$ defined by: $\check\omega_2 \stackrel{R_t}{\sim} \check\omega_1$ holds if and only if $\check\omega_2(0-)=\check\omega_1(0-)$ and $\check\omega_2(s) = \check\omega_1(s)$ for all $s < t$,  for each $t \geq 0$.  Recall that $\check\F_{0-}=\B(\E)\otimes \{\emptyset,\Omega\}$.
\end{rem}

The following corollary follows by selecting $t=\tau(\omega_1)$.

\begin{cor}[Corollary~1 of \cite{cour:65} modified] \label{corollary1}
If $\tau$ is an $\{{\cal F}_{t-}\}$-stopping time, the relation $\omega_2(s) = \omega_1(s)$ for all $s< \tau(\omega_1)$ implies that $\tau(\omega_2) = \tau(\omega_1)$. 
\end{cor}

We now give the theorem which establishes the model for the random effect, impulse-controlled process corresponding to a nominal impulse policy.

\begin{thm} \label{prob-construction}
Let $(\tau,Z)$ be a nominal impulse policy.  For each $k\in \NN$, define the pre-impulse location $Y_k = X(\tau_k-)$ with the nominal impulse being $Z_k$ on the set $\{\tau_k < \infty\}$.  Set $\tau_0=0$.  Then there exists a family of probability measures $\{\PP_x^{(\tau,Z)}: x\in \E\}$ on $(\check\Omega,\check\F)$ under which the coordinate process $X$ satisfies the following properties:  
\begin{itemize}
\item[(a)] under $\PP^{(\tau,Z)}_x$ for each $x\in \E$, $X(0-)=x$ a.s.~and moreover,  for each $k\in \NN$, 
\begin{itemize} 
\item[(i)] $X$ is the fundamental Markov process on the interval $[\tau_{k-1},\tau_k)$;  
\item[(ii)]  on the set $\{\tau_k < \infty\}$, $Q_{(Y_k,Z_k)}$ is a regular conditional distribution of $X(\tau_k)$ given ${\cal F}_{\tau_k-}$; and
\end{itemize} 
\item[(b)] for each $F \in \check\F$, the mapping $x\mapsto \PP_x^{(\tau,Z)}(F)$ is universally measurable.
\end{itemize}
\end{thm}

\begin{proof}

Let $(\tau,Z)$ be a nominal impulse policy.  We build the model by iteratively adding random effect impulse interventions through transition functions.  This is facilitated by lifting the interventions and cycles to a product space.  To identify the components of the product space, define $(\Omega_0,\FFF_0) = (\check\Omega,\check\F)$ as well as the sequence of measurable spaces $\{(\Omega_k,\mathscr{F}_k): k\in \NN\}$ with each $(\Omega_k,\mathscr{F}_k) = (\Omega,\F)$.  Denote an element by $\omega_k \in \Omega_k$ for $k\in \NN_0$.  

We initially construct transitions from paths $(\omega_0,\ldots,\omega_k)$ to $\omega_{k+1}$.  Each additional intervention uses a transition function from $\QQ$ to pick an initial position $v_{k+1}$ for the new path $\omega_{k+1}$, with the measure $\PP_{v_{k+1}}$ then determining the probabilities on $\Omega_{k+1}$.  Thus the transitions from $(\omega_0,\ldots,\omega_k)$ to $\omega_{k+1}$ involve two types of transition functions.  For clarity, we refer to the selection of the initial position using a transition function but the transition to a new path is through a transition kernel.  Following an application of the Ionescu Tulcea extension theorem, we then focus on the initial position $x=\omega_0(0-)$.

Care needs to be taken when an intervention time is infinite for any component so that the transition kernel selects a path for the next component.  A fixed element $\ovl{e} \in {\cal E}$ will be the initial position for the new path and similarly, a fixed impulse control decision $\ovl{z}\in {\cal Z}$ will be the selected impulse for the next component, even though they will never be used  to define the impulse-controlled process.

\noindent
{\em First Intervention.}\/  To define the transition kernel $P_1:\Omega_0\times \mathscr{F}_1 \rightarrow [0,1]$, we begin by carefully defining the transition function $Q_1$ which selects the new initial position for the path $\omega_1\in\Omega_1$.  Set $X_0$ to be the (extended) coordinate process on $\Omega_0$.  Since $\Omega_0=\check\Omega$, define the mapping $T_0: \Omega_0 \rightarrow \check\Omega$ to be $T_0(\omega_0) = \omega_0$.  Now define the first intervention $(\wdt\tau_1,\wdt{Z}_1)$ on $\Omega_0$ by $\wdt\tau_1(\omega_0) := \tau_1(T_0(\omega_0)) = \tau_1(\omega_0)$ and $\wdt{Z}_1(\omega_0) := Z_1(T_0(\omega_0))I_{\{\omega_0\in \Omega_0: \ttau_1(\omega_0)<\infty \}} + \ovl{z} I_{\{\omega_0\in \Omega_0: \ttau_1(\omega_0)=\infty \}}.$   Using $\wdt\tau_1$ define the random variable $\wdt{Y}_1: \Omega_0 \rightarrow {\cal E}$ by $\wdt{Y}_1(\omega_0) := X_0(\wdt\tau_1(\omega_0)-,\omega_0) I_{\{\omega_0\in \Omega_0: \ttau_1(\omega_0)<\infty \}} + \ovl e I_{\{\omega_0\in \Omega_0: \ttau_1(\omega_0)=\infty \}}.$ 
As with $\wdt{Z}_1$, the definition of $\wdt{Y}_1$ on $\{\wdt\tau_1=\infty\}$ is merely for completeness.

Now define the transition function $Q_1: \Omega_0\times {\cal B}({\cal E}) \rightarrow [0,1]$ such that for each $\omega_0\in \Omega_0$ and $G\in {\cal B}({\cal E})$, 
\begin{equation} \label{Q1-def}
Q_1(\omega_0,G) =Q_{(\wdt{Y}_1(\omega_0),\wdt{Z}_1(\omega_0))}(G) I_{ \{\omega_0\in \Omega_0: \wdt\tau_1(\omega_0)<\infty\} } + \delta_{\{\ovl{e}\}}(G) I_{ \{\omega_0\in \Omega_0: \wdt\tau_1(\omega_0)=\infty\} }. 
\end{equation}
Notice that on the set $\{\wdt\tau_1<\infty\}$, for each $\omega_0$, $Q_{(\wdt{Y}_1(\omega_0),\wdt{Z}_1(\omega_0))} \in {\cal P}({\cal E})$ and for each $G \in {\cal B}({\cal E})$, the mapping $\omega_0 \mapsto Q_{(\wdt{Y}_1(\omega_0),\wdt{Z}_1(\omega_0))}(G)$ is $\mathscr{F}_0$-measurable so $Q_1$ is a transition function as claimed.   

Using $Q_1$ and $\PP_x$, the transition kernel $P_1:\Omega_0\times \mathscr{F}_1 \rightarrow [0,1]$ is defined by 
\begin{equation} \label{trans-fn-1}
 P_1(\omega_0,F_1) = \int_{\E} \PP_{v_1}(F_1)\, Q_1(\omega_0,dv_1), \qquad F_1\in \mathscr{F}_1.
\end{equation}

\noindent
{\em Second Intervention.}\/   Denote points in $\Omega_0\times \Omega_1$ by $\mathbf{e}_1=(\omega_0,\omega_1)$.  The first task is to extend the definition of the initial nominal impulse $(\wdt\tau_1,\wdt{Z}_1)$ to $\Omega_0\times \Omega_1$.  First, let $\pi_0: \Omega_0\times \Omega_1 \rightarrow \Omega_0$ be the projection mapping.  Next define $(\wdt\tau_1(\mathbf{e}_1),\wdt{Z}_1(\mathbf{e}_1)) = (\wdt\tau_1(\pi_0(\mathbf{e}_1)),\wdt{Z}_1(\pi_0(\mathbf{e}_1)))=(\wdt\tau_1(\omega_0),\wdt{Z}_1(\omega_0))$ and observe that the slight abuse of notation is reasonable since $(\wdt\tau_1,\wdt{Z}_1)$ is consistently defined on $\Omega_0$ and $\Omega_0\times \Omega_1$.  
Now let  $T_1: \Omega_0\times \Omega_1\rightarrow \Omega_0$ be the mapping such that $\ovl{\omega}_1(0-) = \omega_0(0-)$ and $ \ovl{\omega}_1(t):=T_1(\mathbf{e}_1)(t) = \omega_0(t) I_{\{ 0 \leq t < \wdt\tau_1(\omega_0)\}} + \omega_1(t-\wdt\tau_1(\omega_0))I_{\{t \geq \wdt\tau_1(\omega_0)\}}.$

We now define the random time $\wdt\tau_2:\Omega_0\times \Omega_1 \rightarrow \ovl{\RR}_+$ and nominal impulse $\wdt{Z}_2$ on $\Omega_0\times \Omega_1$.  First, define the projection operator $\wdt{T}: \Omega_0 \to \Omega$ so that $\wdt{T}(\omega_0(0-),\omega_0(\cdot)) = \omega_0(\cdot)$ for all $\omega_0\in \Omega_0$.  Now respectively define $\wdt\tau_2$ and $\wdt{Z}_2$ by 
\begin{equation} \label{wdt-tau2}
\wdt\tau_2(\mathbf{e}_1) = \tau_{2}(\wdt{T}(\ovl{\omega}_{1})) \quad \mbox{and} \quad 
\wdt{Z}_2(\mathbf{e}_1) = Z_{2}(\wdt{T}(\ovl{\omega}_{1})) I_{\{\wdt\tau_2(\mathbf{e}_1) < \infty \}} + \ovl z I_{\{\wdt\tau_2(\mathbf{e}_1) =  \infty \}}.
\end{equation}

We now make an important observation relating $\wdt\tau_1(\omega_0)$ and $\wdt\tau_1(\ovl\omega_1)$ since $\omega_0, \ovl\omega_1\in \check\Omega=\Omega_0$.  By definition $\ovl\omega_1(0-) = \omega_0(0-)$ and $\ovl\omega_1=\omega_0$ on $[0,\tau_1(\omega_0))$.  So since $\tau_1$ is an $\{\check\F_{t-}\}$-stopping time, \corref{corollary1} shows that $\tau_1(\ovl\omega_1)=\tau_1(\omega_0)$ and hence for each $\ovl{\mathbf{e}}_1=(\ovl\omega_1,\omega_1) \in \Omega_0\times \Omega_1$, 
$\wdt\tau_1(\ovl{\mathbf{e}}_1) = \wdt\tau_1(\omega_0).$
Thus this shift from using $\omega_0$ to $\ovl\omega_1$ to define the stopping time $\wdt\tau_2$ does not affect the value of the stopping time $\wdt\tau_1$.  

Using the coordinate process $X_1$ in the component space $\Omega_1$, define the mapping $\wdt{Y}_2: \Omega_0\times \Omega_1 \rightarrow {\cal E}$ by 
\begin{equation} \label{wdt-Y2}
\wdt{Y}_2(\mathbf{e}_1) = X_1(\,(\wdt\tau_2(\mathbf{e}_1)-\wdt\tau_1(\mathbf{e}_1))-\,,\omega_1) I_{ \{\wdt\tau_2(\mathbf{e}_1) < \infty\}} + \ovl e I_{ \{\wdt\tau_2(\mathbf{e}_1) = \infty\}}.
\end{equation}
In particular, on the set $\{\wdt\tau_2<\infty\}$ we observe that $\wdt{Y}_2\in {\cal E}$ corresponds to the value given by the pre-intervention location $\ovl{\omega}_1(\tau_2-)\in {\cal E}$ while $\wdt{Z}_2\in {\cal Z}$ corresponds to the nominal impulse control $Z_2(\wdt T(\ovl \omega_{1}))\in {\cal Z}$.  

Similarly as in \eqref{Q1-def}, define the transition function $Q_2: (\Omega_0\times \Omega_1)\times {\cal B}({\cal E}) \rightarrow [0,1]$ such that for each $\mathbf{e}_1\in \Omega_0\times \Omega_1$ and $G\in {\cal B}({\cal E})$,
\begin{equation} \label{Q2}
Q_2(\mathbf{e}_1,G) =Q_{(\wdt{Y}_2(\mathbf{e}_1),\wdt{Z}_2(\mathbf{e}_1))}(G)I_{ \{\mathbf{e}_1: \wdt\tau_2(\mathbf{e}_1)<\infty\}} + \delta_{\{\ovl{e}\}}(G) I_{ \{\mathbf{e}_1: \wdt\tau_2(\mathbf{e}_1)=\infty\}}. 
\end{equation}
The transition kernel $P_2:(\Omega_0\times \Omega_1)\times \mathscr{F}_2 \rightarrow [0,1]$ is specified by 
$$ \vspace{-1pt}  P_2(\mathbf{e}_1,F_2) = \int_\E \PP_{v_2}(F_2)\, Q_2(\mathbf{e}_1,dv_2), \qquad \mathbf{e}_1\in \Omega_0\times \Omega_1, F_2\in \mathscr{F}_2. $$

\noindent
{\em Induction Step.}\/  We now summarize the definition of the transition kernel $P_{k+1}:\prod_{i=0}^k \Omega_i \times \mathscr{F}_{k+1} \rightarrow [0,1]$ for $k \geq 1$.  To simplify notation, set $\mbf{e}_0 := \omega_0$ and for $k\in \NN$ let $\mathbf{e}_k=(\omega_0,\omega_1,\ldots,\omega_k) \in \prod_{i=0}^k \Omega_i$ and set $\pi_{k-1}:\prod_{i=0}^k \Omega_i \rightarrow \prod_{i=0}^{k-1} \Omega_i$ to be the projection mapping.  Similarly as in the second intervention, for $i=1,\ldots,k$, extend the definitions of the nominal impulses $(\wdt\tau_i,\wdt{Z}_i)$ to $\prod_{i=0}^k \Omega_i$ by setting $(\wdt\tau_i(\mathbf{e}_k),\wdt{Z}_i(\mathbf{e}_k)) = (\wdt\tau_i(\pi_{k-1}(\mathbf{e}_k)),\wdt{Z}_i(\pi_{k-1}(\mathbf{e}_k)))$.  Also, define $\wdt\tau_0 = 0$.

Now let $T_k: \prod_{i=0}^k \Omega_i \rightarrow \check\Omega$ be the mapping defined by $\ovl\omega_k(0-) = \omega_0(0-)$ and 
\setlength{\arraycolsep}{0.5mm}
\begin{equation} \label{partially-pasted-path} 
\begin{array}{rcl}
\ovl{\omega}_k(t) := T_k(\mathbf{e}_k)(t) &=& \left\{\begin{array}{ll}
\omega_i(t-\wdt\tau_i(\mathbf{e}_k)), & \; \wdt\tau_i(\mathbf{e}_k) \leq t < \wdt\tau_{i+1}(\mathbf{e}_k), \; i\leq k-1, \\
\omega_k(t-\wdt\tau_k(\mathbf{e}_k)), & \; t \geq \wdt\tau_k(\mathbf{e}_k).
\end{array}\right.
\end{array}
\end{equation}
As before by \corref{corollary1}, shifting to $\ovl\omega_k$ for the definition of $\wdt\tau_{k+1}$ below does not affect the previous intervention times.

Next define the nominal impulse $(\wdt\tau_{k+1},\wdt{Z}_{k+1})$ on $\prod_{i=0}^k \Omega_i$ similarly as in \eqref{wdt-tau2} and the mapping $\wdt{Y}_{k+1}: \prod_{i=0}^k \Omega_i \rightarrow \E$ using the coordinate process $X_k$ similarly as in \eqref{wdt-Y2} so that on the set $\{\wdt\tau_{k+1} < \infty\}$, $\wdt{Y}_{k+1}\in \E$ corresponds to $\ovl\omega_k(\tau_{k+1}-)\in {\cal E}$ and $\wdt{Z}_{k+1}\in {\cal Z}$ corresponds to $Z_{k+1}(\wdt T(\ovl\omega_k))\in {\cal Z}$.  Define the transition function $Q_{k+1}: (\prod_{j=0}^k \Omega_j )\times {\cal B}({\cal E}) \rightarrow [0,1]$ similarly to \eqref{Q2} in order to define the transition kernel $P_{k+1}: (\prod_{j=0}^k \Omega_j) \times \mathscr{F}_{k+1} \rightarrow [0,1]$ for $\mathbf{e}_k \in \prod_{j=0}^k \Omega_j$ and  $F_{k+1}\in \mathscr{F}_{k+1}$ by 
\vspace{-2pt}\begin{equation} \label{trans-fn-k}
P_{k+1}(\mathbf{e}_k,F_{k+1}) = \int_\E \PP_{v_{k+1}}(F_{k+1})\, Q_{k+1}(\mathbf{e}_k,dv_{k+1}).
\end{equation}

\vspace{-4pt}
This inductive process defines a sequence of transition kernels $\{P_k; k\in \NN\}$ which meets the conditions of the Ionescu Tulcea Extension Theorem (see Proposition V.1.1 (page 162) of \cite{neve:65}). 
 As a result, for each $\omega_0\in \Omega_0$ there exists a unique measure $\wdt\PP^{(\ttau,\wdt{Z})}_{\omega_0}$ on $(\wdt\Omega,\wdt{\cal G}):=(\prod_{j=0}^\infty \Omega_j,\otimes_{j=0}^\infty \FFF_j)$ whose value for every measurable rectangle $\prod_{k=0}^N F_k$ with $N\in \NN_0$ is \vspace{-6pt}
 \begin{displaymath} \vspace{-4pt} 
\wdt\PP^{(\ttau,\wdt{Z})}_{\omega_0}(\prod_{k=0}^N F_k)   = I_{F_0}(\omega_0) \int_{F_1} \int_{F_2}\cdots\int_{F_N} P_{N}(\mathbf{e}_{N-1},d\omega_N)\ldots P_2(\mathbf{e}_1,d\omega_2)\,  P_1(\omega_0,d\omega_1) 
\vspace{-3pt} \end{displaymath}
and such that for every nonnegative random variable $Y$ on $(\wdt\Omega,\wdt{\cal G})$, the expression
$\int_{\wdt\Omega} Y(\wdt\omega)\, \wdt\PP^{(\ttau,\wdt{Z})}_{\omega_0}(d\wdt\omega)$
is an $\mathscr{F}_0$-measurable function of $\omega_0$.

Now recall the extension of $\PP_x$ on $(\Omega,\F)$ to $\check\PP_x$ on $(\check\Omega,\check\F)$ and that $(\Omega_0,\mathscr{F}_0) = (\check\Omega,\check\F)$.  Define a family of measures $\{\wdt\PP^{(\ttau,\wdt{Z})}_x: x \in \E\}$ such that 
$$\wdt\PP^{(\ttau,\wdt{Z})}_x(G) = \int_{\Omega_0} \wdt\PP^{(\ttau,\wdt{Z})}_{\omega_0}(G)\, \check\PP_x(d\omega_0), \qquad \forall G\in \wdt{\cal G}.$$ 
We determine the measurability of the mapping $x\mapsto \wdt\PP^{(\ttau,\wdt{Z})}_x(G)$ for each $G \in \wdt{\cal G}$.  Observe that for each $F_0\in \mathscr{F}_0$, $\wdt\PP^{(\ttau,\wdt{Z})}_x(F_0) = \check\PP_x(F_0)$ so since $(\Omega_0,\mathscr{F}_0,X_0,\{\check\F_t\},\{\check\PP_x: x\in \E\})$ is a Markov family, $\wdt\PP^{(\ttau,\wdt{Z})}_x(F_0)$ is a universally measurable function of $x$.  It then follows that for any simple function $X(\omega_0) := \sum_{j=1}^n c_j I_{F_j}(\omega_0)$, the function $x\mapsto \sum_{j=1}^n c_j \wdt\PP^{(\ttau,\wdt{Z})}_x(F_j) = \EE^{\wdt\PP^{(\ttau,\wdt{Z})}_x}[X]$ is universally measurable.  This, in turn, implies that the mapping $x\mapsto \EE^{\wdt\PP^{(\ttau,\wdt{Z})}_x}[X]$ is universally measurable for any any nonnegative random variable $X$ on $(\Omega_0,\mathscr{F}_0,\check\PP_x)$. Therefore 
 the function $x\mapsto \wdt\PP^{(\ttau,\wdt{Z})}_x(G)$ is universally measurable since $\wdt\PP^{(\ttau,\wdt{Z})}_{\omega_0}(G)$ is an $\mathscr{F}_0$-measurable random variable.

For any nominal impulse policy $(\tau,Z)$ on $(\Omega,\F)$, the above construction defines a measure $\wdt\PP^{(\ttau,\wdt{Z})}_x$ on $(\wdt\Omega,\wdt{\cal G})$ for which each component coordinate process $X_k$ is the fundamental Markov process, $X_0(0-) = x$ a.s.~and for $k\geq 1$, conditional on ${\cal F}_{\tau_k-}$, $X_k(0)$ has distribution $Q_{(\wdt{Y}_k,\wdt{Z}_k)}$\; $\wdt\PP^{(\ttau,\wdt{Z})}_x$-almost surely.  Notice that the random times $\wdt\tau_k$ are well-defined on $\wdt\Omega$ using the projection of $\wdt\Omega$ onto $\prod_{j=0}^{k-1} \Omega_j$ for each $k\in \NN$.  This probability space can be used to define the random effect impulse controlled process by pasting together the cycles from each component.  Set  $\wdt\tau_0=0$ and define the process $\wdt{X}$ on $(\wdt\Omega,\wdt{\cal G},\wdt\PP^{(\ttau,\wdt{Z})}_x)$ by
\begin{equation} \label{inv-proc-def}
\wdt{X}(t) = X_k(t-\wdt\tau_k) \quad \mbox{for } \wdt\tau_k \leq t < \wdt\tau_{k+1}, \qquad k = 0, 1, 2, \ldots.
\end{equation}
Also define $\wdt{X}(0-,\wdt\omega) = \omega_0(0-)$ for all $\wdt\omega\in \wdt\Omega$.  Since $\wdt{X}$ has right continuous paths and is adapted to the filtration it generates, $\wdt{X}$ is progressively measurable and hence is a measurable mapping of $\wdt\Omega$ into $\check\Omega$.  

Finally, since $\wdt{X}: \wdt\Omega \rightarrow \check\Omega$, define $\PP_x^{(\tau,Z)} := \wdt\PP^{(\ttau,\wdt{Z})}_x \wdt{X}^{-1}$ to be the distribution of $\wdt{X}$ on $\check\Omega$.  
Of particular note is the fact that the construction of the nominal interventions $(\wdt\tau,\wdt{Z})=\{(\wdt\tau_k,\wdt{Z}_k): k\in \NN\}$ is such that they correspond to the original nominal impulse policy $(\tau,Z)=\{\tau_k,Z_k): k\in \NN\}$.  Observe also that $\PP_{\cdot}^{(\tau,Z)}(F)$ as a function of the initial position is universally measurable for each $F\in \F$.  The desired model for the state process under the nominal impulse control policy $(\tau,Z)$ is therefore given by the coordinate process $X$ on $(\check\Omega, \check\F,\PP_x^{(\tau,Z)})$.

  We conclude the proof by showing that on the set $\{\tau_k < \infty\}$,  $k\in \NN$, the transition function $Q_{(Y_k,Z_k)}$ is a regular conditional distribution of $X(\tau_k)$ given $\F_{\tau_k-}$ under the measure $\PP^{(\tau,Z)}_x$; recall  $Y_k = X(\tau_k-)$. In view of the definition of $Q_{(y,z)}$, we only need to show that $\PP^{(\tau,Z)}_x (  \{X(\tau_k)\in B\}|\F_{\tau_{k}-})(\omega) = Q_{(Y_k(\omega),Z_k(\omega))}(B) $   for $\PP^{(\tau,Z)}_x $-almost every $\omega$ and each $B\in \B(\E)$. Note that   on the set $\{\wdt\tau_{k} < \infty \}$, we have $\{\wdt X(\wdt\tau_{k}) \in B\} = \{ X_{k}(0) \in B\} \in \mathscr F_{k}$. Thus 
  \vspace{-4pt}
\begin{displaymath} P_{k}(\mbf e_{k-1}, \{\wdt X(\wdt\tau_{k}) \in B\} )= \int_{\mathcal E} \PP_{v_{k}}  \{ X_{k}(0) \in B\} Q_{k}(\mbf e_{k-1}, dv_{k}) 
= Q_{(\wdt Y(\mbf e_{k-1}), \wdt Z(\mbf e_{k-1}))}(B).\vspace{-4pt}
\end{displaymath}  
Let $A \in \F_{\tau_k-}$ and define $\wdt{A} = \wdt{X}^{-1}(A) \in \otimes_{i=0}^{k-1} \mathscr F_i$. Then, using $\pmb \int_{{\!\!\mathmakebox[1em][l]{\wdt A} }}   : = \idotsint_{\wdt{A}} $,   
 \begin{align*}
& \PP^{(\tau,Z)}_x (A\cap \{X(\tau_k)\in B\}) = \wdt\PP^{(\ttau,\wdt{Z})}_x(\wdt{A} \cap \{\wdt{X}(\wdt\tau_{k})\in B\}) \\[-2pt]
 & =
{\pmb \int}_{{\!\!\!\mathmakebox[1em][l]{\wdt A} }}  P_{k}(\mbf e_{k-1}, \{\wdt X(\wdt\tau_{k}) \in B\} )\, P_{k-1}(\mbf{e}_{k-2},d\omega_{k-1})\cdots P_1(\omega_0,d\omega_1)\, \PP_x(d\omega_0) \\[-1.0pt]
& = {\pmb \int}_{{\!\!\!\mathmakebox[1em][l]{\wdt A} }}  Q_{(\wdt Y(\mbf e_{k-1}), \wdt Z(\mbf e_{k-1}))}(B)\, P_{k-1}(\mbf{e}_{k-2},d\omega_{k-1})\cdots P_1(\omega_0,d\omega_1)\, \PP_x(d\omega_0) \\[-1.0pt]
& = {\pmb \int}_{{\!\!\!\mathmakebox[1em][l]{\wdt A} }}  Q_{(\wdt Y(\mbf e_{k-1}), \wdt Z(\mbf e_{k-1}))}(B)\, \wdt\PP^{(\ttau,\wdt{Z})}_x(d\wdt\omega) 
=\int_A Q_{(Y_k(\omega),Z_k(\omega))}(B)\, \PP^{(\tau,Z)}_x(d\omega),
\end{align*} 
establishing that $Q_{(Y_k,Z_k)}$ is a regular conditional distribution of $X(\tau_k)$ given $\F_{\tau_k-}$.
\end{proof}

\begin{cor} \label{prob-construction-nu}
Let $(\tau,Z)$ be a nominal impulse policy with corresponding family of probability spaces $\{(\check\Omega,\check\F,\PP^{(\tau,Z)}_x):x\in \E\}$ and let $\nu \in {\cal P}(\E)$.  Define $\PP^{(\tau,Z)}_\nu$ by $\PP^{(\tau,Z)}_\nu(F) = \int_\E \PP^{(\tau,Z)}_x(F)\, \nu(dx)$ for each $F\in \check\F$.  Then the coordinate process $X$ on $(\check\Omega,\check\F,\PP^{(\tau,Z)}_\nu)$ satisfies \thmref{prob-construction}(a) but with $X(0-)$ having distribution $\nu$.
\end{cor}

\begin{rem}
The construction of the models in \thmref{prob-construction} and \corref{prob-construction-nu} can be extended to policies $(\tau,Z)$ for which the intervention times $\{\tau_k\}$ are stopping times with respect to the universal completion of the filtration $\{\F_{t-}\}$.  Only minor adjustments are required in the proofs of this and the next section.  The adaptation of the Markov results in \sectref{sect-Markov} would require more effort so this extension is not pursued.
\end{rem}

 We next identify a class of nominal impulse policies such that the controlled process $X$ has independent cycles.  

\section{Policies with Independent Cycles}\label{sect-ind-cycles}
\thmref{prob-construction} establishes the existence of a family of measures $\{\PP^{(\tau,Z)}_x: x\in \E\}$ on $(\check\Omega,\check\F)$ by first defining a sequence of transition kernels $\{P_{k+1}: k\in\NN_0\}$ in \eqref{trans-fn-k} on the product space $(\wdt\Omega,\wdt{\cal G})$ which sequentially select the next cycle.  These transition kernels in turn depend on random effects transition functions $\{Q_{k+1}: k\in \NN_0\}$ with  
\begin{equation} \label{Q-k1-def}
Q_{k+1}(\mbf{e}_k,\cdot) = I_{\{\ttau_{k+1}(\mbf{e}_k)<\infty\}} Q_{(\wdt{Y}_{k+1}(\mbf{e}_k),\wdt{Z}_{k+1}(\mbf{e}_k))}(\cdot) + I_{\{\ttau_{k+1}(\mbf{e}_k)=\infty\}} \delta_{\ovl{e}}(\cdot).
\end{equation}
Recall at the end of the proof of \thmref{prob-construction} it was observed that the $k$th random effect distribution is given by $Q_{(Y_{k},Z_k)}$, in which  $Y_{k} = X(\tau_k-)$.  This structure enables a class of policies to be identified for which the controlled process $X$ has independent cycles; that is, for which $\{X(t): 0 \leq t < \tau_1\}$, $\{X(t): \tau_1 \leq t < \tau_2\}$, \ldots, are independent stochastic processes.

To properly define the subclass of  nominal policies for independent cycles, we need a subtle adjustment to the standard shift operator to account for the augmented path space $\check\Omega$.  Recall that $\check\omega\in \check\Omega$ has elements $\check\omega=(\check\omega(0-),\check\omega(\cdot))$. 
For $t > 0$, let $\theta_t:\check\Omega \rightarrow \Omega$ denote the shift operator such that $\theta_t(\check\omega) = \check\omega(t+\cdot)$ for all $\check\omega \in \check\Omega$; note this operator cuts off the path $\{\check\omega(r): 0 \leq r < t\}$ as well as $\check\omega(0-)$ from $\check\omega$.  
For $t=0$, $\theta_t$ is the identity operator on $\check\Omega$.  Similarly, for a stopping time $\eta$, define the random shift operator $\theta_{\eta}:\check\Omega\rightarrow \Omega$ by $\check\omega(s) \mapsto \check\omega(\eta(\check\omega)+s)$ for $s\geq 0$ on the set $\{\eta < \infty\}$; the random shift operator is undefined on $\{\eta=\infty\}$.   

Since we set $\tau_0=0$, it will be helpful to define   the special symbol $\theta_{\tau_0} $ to represent  the identity mapping on $ \check\Omega$ so that  single relations can be used for every $k\in \NN_{0}$ in Definitions \ref{independent-cycles-nominal-impulse-policy} and \ref{Markov-nominal-impulse-policy}.  The symbol $\theta_{\tau_0} $  indicates the absence of any shift.  Note that, in contrast with $\theta_\eta$ for a general stopping time $\eta$, the augmented value $\check\omega(0-)$ is retained by $\theta_{\tau_0}(\check\omega)$.  Now for $r\geq 0$ introduce the extended shift operator $\check\theta_r: \check\Omega \to \check\Omega$ with $\check\theta_r(\check\omega) = (\check\omega(r-),\check\omega(r+\cdot))$ and define the symbol $\theta_{\tau_0+r} := \check\theta_r\circ\theta_{\tau_0}$ for any $r \ge 0$.   

\begin{defn}[Independent Cycles Nominal Impulse Policy] \label{independent-cycles-nominal-impulse-policy}
Set $\tau_0=0$.  An independent-cycles nominal impulse policy is a nominal impulse policy $(\tau,Z) $ 
 for which for each $k\in \NN$: \begin{itemize} \item[(a)] there exists a random time $\sigma_{k}$ such that $\tau_{k} = \tau_{k-1} + \sigma_{k}\circ\theta_{\tau_{k-1}}$ on the set $\{\tau_{k-1} < \infty\}$; and \item[(b)] on the event $\{\tau_k < \infty\}$, the intervention $(\tau_{k},Z_{k})$ is such that $Q_{(Y_{k}(\omega),Z_k(\omega))}$ does not depend on $\omega$.\end{itemize}
\end{defn}

Policies $(\tau,Z)$ satisfying \defref{independent-cycles-nominal-impulse-policy} have two common features.  First the decision when to intervene is only dependent on the current cycle.  Also, the random effects distributions for the interventions are deterministic in the sense that, for the $k$th intervention, the impulse $Z_{k}(\omega)$ is constant as a function of $\omega$ and either the random effects distribution $Q_{(Y_{k}(\omega),Z_k(\omega))}$ is independent of the intervention location $Y_k=X(\tau_{k}-)$ or the intervention occurs only when each path hits the same state.


The following proposition establishes the independence of the cycles for $(\tau,Z)$ satisfying \defref{independent-cycles-nominal-impulse-policy}. 

\begin{prop}\label{prop6.2}
Let $\nu\in {\cal P}(\E)$, $(\tau,Z)$ be an independent cycles nominal impulse policy and let $(\check\Omega,\check\F,\PP^{(\tau,Z)}_\nu)$ and $X$ be the probability space and coordinate process of \corref{prob-construction-nu}, respectively.  Then the cycles $\{X(t): \tau_k \leq t < \tau_{k+1}\}$ for $k\in \NN_0$ are independent.
\end{prop}

\begin{proof}
 Let $x\in \E$ be chosen arbitrarily and let $(\tau,Z)$ satisfy \defref{independent-cycles-nominal-impulse-policy}.  It suffices to prove the result for $\PP^{(\tau,Z)}_x$.  The first goal is to establish that each measure in the family $\{\wdt\PP^{(\htau,\wdh{Z})}_x: x\in \E\}$ from the proof of \thmref{prob-construction} is a product measure on $(\wdt\Omega,\wdt\G)$.  As a consequence, the component paths used to define the cycles of $\wdt{X}$ in \eqref{inv-proc-def} are independent processes.  

One   adjustment to the construction of $\wdt\PP^{(\htau,\wdh{Z})}_x$ is helpful.  By \defref{independent-cycles-nominal-impulse-policy}(b), each $Q_{(X(\tau_{k+1}(\omega)-,\omega),Z_{k+1}(\omega))}$ is constant in $\omega$.  In \eqref{Q-k1-def}, the measure $\delta_{\ovl{e}}(\cdot)$ is arbitrarily chosen as the transition function to a new initial point $\ovl{e}$ for the next component path $\omega_{k+1}$ when $\ttau_k$ is infinite.  For the construction of $\wdt\PP^{(\htau,\wdh{Z})}_x$ in this section, replace $\delta_{\ovl{e}}$ by the fixed transition function $Q_{(X(\tau_{k+1}(\omega)-,\omega),Z_{k+1}(\omega))}$ so that  $Q_{k+1}(\mbf{e}_k,\cdot) = Q_{(\wdt{Y}_{k+1}(\mbf{e}_k),\wdt{Z}_{k+1}(\mbf{e}_k))}(\cdot) = Q_{k+1}(\cdot)$ does not depend on $\mbf{e}_k$.  Referencing \eqref{wdt-Y2} for the pattern of the definition of $\wdt{Y}_k$, since $\ttau_k(\mbf{e}_{k-1}) = \sum_{i=1}^k \sigma_i(\omega_{i-1})$,  it follows that $\wdt{Y}_k(\mbf{e}_{k-1}) = \wdt{Y}_k(\omega_{k-1})$. 
 
Consider $k=1$.  Again, let $\wdt\PP^{(1)}_x$ denote the marginal of $\wdt\PP^{(\htau,\wdh{Z})}_x$ on $(\Omega_0\times \Omega_1, \FFF_0\otimes \FFF_1)$.  Now arbitrarily pick $F_0\in \FFF_0$ and $F_1\in \FFF_1$.  Combining \defref{independent-cycles-nominal-impulse-policy}(b) with the definition of the transition kernel $P_1$ in \eqref{trans-fn-1}, we have
\setlength{\arraycolsep}{0.5mm}
\begin{eqnarray*}
\wdt\PP^{(\htau,\wdh{Z})}_x(F_0\times F_1) = \wdt\PP^{(1)}_x(F_0\times F_1) &=& \int_{F_0} \int_{\E_1} \PP_{v_1}(F_1)\, Q_{(\wdt{Y}_1(\omega_0),\wdt{Z}_1(\omega_0))}(dv_1)\, \PP_x(d\omega_0) \\
&=& \PP_x(F_0) \cdot \int_{\E_1} \PP_{v_1}(F_1)\, Q_1(dv_1)
\end{eqnarray*}
and it follows that $X_0$ and $X_1$ are independent.  An induction argument now establishes the claim that $\wdt\PP^{(\htau,\wdh{Z})}_x$ is a product measure on $(\wdt\Omega,\wdt\G)$ and hence all coordinate processes $\{X_k: k\in \NN_0\}$ are independent.  Again by \defref{independent-cycles-nominal-impulse-policy}(a) for each $k\in \NN$, $\sigma_k$ only depends on $\omega_{k-1}$ and so the snippets of the coordinate processes $\{X_0(t,\omega_0): 0\leq t < \sigma_1(\omega_0)\}$, $\{X_1(t,\omega_1): 0 \leq t < \sigma_2(\omega_1)\}$, $\{X_2(t,\omega_2): 0 \leq t < \sigma_3(\omega_2)\}$, \ldots are independent stochastic processes.  The result now follows from the definition of $\wdt{X}$ in \eqref{inv-proc-def} and its relation to $X$ defined on $(\Omega,\F,\PP^{(\tau,Z)}_x)$.
\end{proof}

An immediate consequence is that when the same rule for the random time $\sigma$ is used for each intervention $\tau_{k+1}$ in \defref{independent-cycles-nominal-impulse-policy} along with the same random effects distribution $Q$, the cycles following the first are iid.  

\begin{cor} \label{iid-cycles}
Let $\nu\in {\cal P}(\E)$, $(\tau,Z)$ be an independent cycles nominal impulse policy such that for each $k\in\NN$, $\sigma_{k}=\sigma$ and $Q_{k}(\cdot)=Q(\cdot)$ for some fixed random time $\sigma$ and distribution $Q$ on $\E$, respectively.  Let $(\check\Omega,\check\F,\PP^{(\tau,Z)}_\nu)$ and $X$ be the probability space and coordinate process of \corref{prob-construction-nu}, respectively.  Then the cycles $\{X(t): \tau_k \leq t < \tau_{k+1}\}$ for $k\in \NN_0$ are independent and the cycles for $k\in \NN$ are identically distributed.
\end{cor} 

A benefit of this result is that for $(\tau,Z)$ satisfying \corref{iid-cycles}, one is able to employ renewal arguments when solving random-effects impulse control problems having long-term average criteria.

\section{Stationary Markov Nominal Impulse Policies}\label{sect-Markov}

We now examine the Markov nature of the coordinate process $X$. 
 In Proposition~V.2.1 of \cite{neve:65} (see p.~168),  the Markov property is established for models in which the transition function only depends on the current state, not the entire past.  Intuitively, to mimic Proposition~V.2.1 of \cite{neve:65} to establish the Markov property of the process, at the minimum the intervention decision for our construction will need to be based solely on the process during the current cycle  similar to \defref{independent-cycles-nominal-impulse-policy}(a).  In fact, stronger conditions are required and we define a class of {\em stationary}\/ policies.

The idea of a stationary   policy is that the same rule for an intervention is used at each time.  Recall, the initial intervention is defined on the augmentated space $\check\Omega$ in order to allow $\tau_1=0$ whereas all later interventions only use $\Omega$.  
To simplify \defref{Markov-nominal-impulse-policy}(a), it is necessary to first project $\check\Omega$ to $\Omega$ before applying the cycle intervention time $\sigma$.  
Recall, the operator $\wdt{T}: \check\Omega \to \Omega$ has been defined so that for each $\check\omega=(\check\omega(0-),\check\omega(\cdot))$, $\wdt{T}(\check\omega) = \check\omega(\cdot)$; see the second intervention in the proof of \thmref{prob-construction}.  When $k=1$ in \defref{Markov-nominal-impulse-policy}(a), $\sigma$ is to be interpreted as the composition $\sigma\circ \wdt{T}$.


Furthermore, one of the challenges in the ensuing analysis 
is relating ``clock times'' to their corresponding times in whichever cycles they occur.  Throughout this section, $s$ and $t$ will represent clock times from the beginning whereas $u$ and $v$ will represent the time within the appropriate cycles.

\begin{defn}[Stationary Markov Nominal Impulse Policy] \label{Markov-nominal-impulse-policy}
A stationary Markov nominal impulse policy is a nominal impulse policy $(\tau,Z)= \{(\tau_k,Z_k): k\in \NN\}$ for which 
 there exist measurable functions $\sigma: \Omega \rightarrow (0,\infty]$ and $\mathfrak{z}:\E\rightarrow \Z$ such that: 
 \begin{itemize}
\item[(a)] for each $k \ge 1$,  $ \tau_k = \tau_{k-1} + \sigma \circ\theta_{\tau_{k-1}}$; 
 and on the event $\{\tau_{k-1}<\infty\}$, for each $u \geq 0$, 
\begin{equation} \label{markov-stopping-time-cnd}
\{\sigma\circ\theta_{\tau_{k-1}} > u\}  \ \subset \  \{ \sigma\circ\theta_{\tau_{k-1}} = u+\sigma\circ\theta_{\tau_{k-1}+u} \};
\end{equation}  
 and 
 \item[(b)] $Z_k = \mathfrak{z}(X(\tau_k-))$.
\end{itemize}
\end{defn}

The random time 
$ \sigma \circ\theta_{\tau_{k-1}} $ gives the length of the $k$th cycle.  The condition \eqref{markov-stopping-time-cnd} implies that $  \sigma \circ\theta_{\tau_{k-1}} $ has the property that if it exceeds time $u$, then it consists of the known elapsed cycle time $u$ plus the random time for the shifted path that only looks to the future.  This condition is a terminal time-like condition for the cycle length $  \sigma \circ\theta_{\tau_{k-1}} $ which is inspired by the definition of a terminal time for general Markov processes given by \cite{meye:75} and \cite{chun:05}.  \defref{Markov-nominal-impulse-policy}(b) requires the   nominal impulse $Z_{k}$ to only depend on the position of the process just prior to the intervention through the given function   $\mathfrak z$.

The intuition underlying a stationary Markov policy is that the same policy should be applied going forward from each time $s\geq 0$, which we   establish  in \propref{stationary-prop} below.

\begin{prop} \label{stationary-prop}
Let $(\tau,Z)$ be a stationary Markov nominal impulse policy.  For each $s \geq 0$ and $k\in \NN$, let $(\tau_k^{(s)},Z_k^{(s)})$ denote the $k$th intervention  following time $s$: 
\begin{displaymath}(\tau^{(s)}_k(\check\omega),Z^{(s)}_k(\check\omega)) = (\tau_{j+k}(\check\omega)-s,Z_{j+k}(\check\omega)),  \quad \mbox{for } \check\omega\in \{\tau_j \leq s < \tau_{j+1}\}, \quad j\in \NN_0.\end{displaymath}
For each $\check\omega \in \check\Omega$, define $\check\omega^{(s)} = \theta_s(\check\omega)$ to be the shifted path.  Then for each $k \in \NN$, 
\begin{displaymath} 
\tau_k^{(s)}(\check\omega) = \tau_k(\check\omega^{(s)}),  \quad \mbox{and} \quad Z_k^{(s)}(\check\omega) = Z_k(\check\omega^{(s)}), 
 \qquad \forall\, \check\omega \in \check\Omega; \end{displaymath}
that is, the $k$th intervention of the policy $(\tau,Z)$ after time $s$ for each $\check\omega \in \check\Omega$ is exactly the same as the $k$th intervention of $(\tau,Z)$ for the shifted path $\check\omega^{(s)}$; in particular, $\tau^{(s)}_k$ does not depend on the number $j$ of interventions up to time $s$.
\end{prop}

\begin{proof}
Arbitrarily fix $s \geq 0$.  Observe that the collection of sets $\{\tau_j \leq s < \tau_{j+1}\}$, $j\in \NN_0$ partitions $\check\Omega$.  Consider $\check\omega \in \{\tau_j \leq s < \tau_{j+1}\}$ for an arbitrary $j\in \NN_0$.  Then $\tau_{j+1}(\check\omega)$ is the first intervention after time $s$ and the length of time $\tau_1^{(s)}(\check\omega)$ from $s$ until this intervention is $\tau_{j+1}(\check\omega)-s$.  Since $u:=s-\tau_j(\check\omega) < \tau_{j+1}(\check\omega)-\tau_j(\check\omega) = \sigma\circ\theta_{\tau_j}(\check\omega)$, \defref{Markov-nominal-impulse-policy}(a,ii) on $\sigma$ implies $\sigma\circ\theta_{\tau_j}(\check\omega) = u + \sigma\circ\theta_{\tau_j+u}(\check\omega) = u + \sigma\circ\theta_s(\check\omega).$  Using this in \defref{Markov-nominal-impulse-policy}(a) yields $\tau_{j+1}(\check\omega) = s + \sigma\circ\theta_s(\check\omega)$.  It now follows that $\sigma(\check\omega^{(s)}) = \tau_{j+1}(\check\omega)-s = \tau^{(s)}_1(\check\omega)$.  Since $\check\omega^{(s)}\in \Omega$ and $(\tau,Z)$ is a stationary Markov policy, $\tau_1(\check\omega^{(s)}) = \sigma(\check\omega^{(s)}) = \tau^{(s)}_1(\check\omega)$.    

Using an induction argument, we consider the $k$th intervention following time $s$ of $\check\omega$. The length of time $\tau_k^{(s)}(\check\omega)$ from $s$ to the $k$th intervention is $\tau_{j+k}(\check\omega)-s$.  Again since $(\tau,Z)$ is a stationary Markov policy and $\sigma\circ\theta_{\tau_j(\check\omega)}(\check\omega) > s-\tau_j(\check\omega)$, we have
\vspace{-2pt}\begin{eqnarray*}
\tau_{j+k}(\check\omega) &=& \tau_{j+k-1}(\check\omega) + \sigma\circ\theta_{\tau_{j+k-1}}(\check\omega) \\[-2pt]
&=& s + (\tau_{j+k-1}(\check\omega)-s) + \sigma\circ\theta_{s+(\tau_{j+k-1}-s)}(\check\omega) \\[-2pt]
&=& s + \tau_{k-1}(\check\omega^{(s)}) + \sigma\circ\theta_{\tau_{k-1}(\check\omega^{(s)})}(\check\omega^{(s)}) = s + \tau_k(\check\omega^{(s)});\vspace{-3pt}
\end{eqnarray*}
\defref{Markov-nominal-impulse-policy}(a) is used on the path $\check\omega^{(s)}\in \check\Omega$ to obtain the last equality.  Therefore $\tau_k(\check\omega^{(s)}) = \tau_{j+k}(\check\omega) - s = \tau_k^{(s)}(\check\omega)$.  

Finally, since for each $k\in \NN$, $\check\omega(\tau_{j+k}(\check\omega)-) = \check\omega^{(s)}(\tau_k(\check\omega^{(s)})-)$, it immediately follows from \defref{Markov-nominal-impulse-policy}(b) that $Z_k^{(s)}(\check\omega) = Z_k(\check\omega^{(s)})$.  
\end{proof}

Our goal is to show 
that  $(\check\Omega,\check\F, X, \{\check \F_{t} \},  \{\PP_{x}^{(\tau, Z)}, x\in \E\})$ is a Markov family corresponding to each stationary Markov nominal impulse policy $(\tau,Z)$, in which $X$ is the coordinate process on $\check\Omega$.
  As in \thmref{prob-construction}, the measure $\PP^{(\tau,Z)}_x$ is a distribution on $  \check\Omega$ of a process $\wdh{X}$ defined on the countable product space $(\wdt\Omega,\wdt{\cal G},\wdh\PP^{(\htau,\wdh{Z})}_x)$.  (In this section, to distinguish the family of measures and process on the product space $(\wdt\Omega,\wdt{\cal G})$ arising from a stationary Markov nominal impulse policy from those for a general nominal impulse policy, we designate the measures as $\{\wdh\PP^{(\htau,\wdh{Z})}_x\}$ and the corresponding 
 process satisfying \eqref{inv-proc-def} by $\wdh{X}$.)   Using the $\wdh \F_{t}$ defined below, our proofs show  that $(\wdt\Omega,\wdt{\cal G},\wdh{X},\{\wdh\F_t\},\{\wdh\PP^{\htau,\wdh{Z})}_x, x\in \E\})$ is a Markov family and that $(\check\Omega,\check\F, X, \{\check \F_{t} \},  \{\PP_{x}^{(\tau, Z)}, x\in \E\})$ then inherits this property.  

The subclass of stationary Markov  nominal impulse policies consists of those policies for which future decisions are independent of the past, given the present.  For each ``present time'' $s \geq 0$, we therefore need to define the filtration of future information; also information just prior to an intervention as well as the information following an intervention are needed in the ensuing analysis.  Most of this analysis occurs on the product space $\wdt\Omega$; we therefore use the following notation: 
\begin{itemize} 
\item $\wdh\F_t = \sigma(\wdh{X}(0-))\vee\sigma(\wdh{X}(r): 0\leq r \leq t)$; 
\item for each $j\in \NN_0$, $\wdh\F_{\htau_{j+1}-} = \sigma(A\cap \{\htau_{j+1} > t\}: A\in \wdh\F_t,\, t\geq 0)$ is the information prior to the $(j+1)$st intervention; and
\item for each $j\in \NN$, the filtration $\{\F^{(j)}_u\}$ of the $j$th coordinate process $X_j$ on $\Omega_j$ has $\F^{(j)}_u = \sigma(X_j(r): 0 \leq r \leq u)$; we also define the filtration $\{\check\F^{(0)}_u\}$ of the $0$th coordinate process $X_0$ with $\check\F^{(0)}_u = \check\F_u$.   (These filtrations arise in the application of the Markov property of $X_j$ on $\Omega_j$ in the proof of \thmref{Markov-family-thm}.)
\end{itemize} 

Our analysis of the Markov  family property involves partitioning $\wdt\Omega$ according to the cycles and examining the trace of the $\sigma$-algebra $\wdh\F_s$ on these cycles.  For each $s \geq 0$ and $j\in \NN_0$, the $\{\htau_j \leq s < \htau_{j+1}\}$-trace $\sigma$-algebra of $\wdh\F_s$ is defined to be $\{A\cap \{\htau_j \leq s < \htau_{j+1}\}: A\in \wdh\F_s\}$.

Three lemmas are needed.  The first one transfers the terminal time-like  condition \eqref{markov-stopping-time-cnd} to the product space $\wdt\Omega$ and  gives a crucial implication.  For its formulation, recall some important definitions in the construction of $\{\wdh\PP^{(\htau,\wdh{Z})}_x\}$ on $(\wdt\Omega,\wdt{\cal G})$ in the proof of \thmref{prob-construction}.  For each $k\in \NN_0$, $\mbf{e}_k = (\omega_0,\ldots,\omega_k) \in \prod_{i=0}^k \Omega_i$, $\ovl\omega_k$ is given by \eqref{partially-pasted-path} and $\htau_{k+1}(\wdt\omega) = \tau_{k+1}(\ovl\omega_k)$ for those $\wdt\omega\in \wdt\Omega$ whose projection onto the first $k+1$ components is $\mbf{e}_k$ and for which $\htau_k(\wdt\omega) < \infty$.  Note  that  $\htau_{k+1}(\wdt\omega)  =  \htau_{k+1}(\mbf e_{k})$ depends only on $\mbf e_{k}$.  For simplicity of exposition in the sequel, when $k=0$, the notation $\htau_{k}(\mbf{e}_{k-1})$ is to be understood to be $\htau_0 = 0$.  

\begin{lem} \label{sigma-term-time-lemma}
Let $(\tau,Z)$ be a stationary Markov nominal impulse policy and for each $k\in\NN$, define $\htau_k$ as in the proof of \thmref{prob-construction}, with $\htau_0=0$.  Then  for each $k\in \NN$,
 on the set $\{\htau_{k-1} < \infty\}$, 
\begin{itemize}
\item[(a)] the length of the  $(k-1)$st cycle   
$\htau_{ k}(\wdt\omega)- \htau_{ k-1}(\wdt\omega) = \sigma(\omega_{ k-1})$ depends only on the component $\omega_{ k-1}$ of $\wdt\omega$, where $\sigma$ is the function given in Definition \ref{Markov-nominal-impulse-policy}; and 
\item[(b)] for $\wdt\omega \in \{\htau_{ k-1} < \infty\}$, for all $u \geq 0$, 
\begin{equation} \label{sigma-term-time}
\wdt\omega \in \{\htau_{k}- \htau_{k-1} > u\} \quad \mbox{implies}\quad \sigma(\omega_{ k-1}) = u + \sigma\circ\theta_{u}(\omega_{ k-1}).
\end{equation}
\end{itemize}
\end{lem}

\begin{rem}  When $k=1$, $\sigma(\omega_{0})$ is to be understood as  $\sigma\circ \wdt T(\omega_{0})$ and
the shift operator $\check\theta_u$ must be used in \eqref{sigma-term-time}. 
This use arises in the proof below from the identity $\theta_{\tau_0+u} = \check\theta_u\circ\theta_{\tau_0} = \check\theta_u$.  For simplicity of exposition, we leave it to the reader to make this substitution when $\theta_{\tau_0}$ is used throughout the remainder of the paper, such as in \eqref{e:htau_j+1} of \corref{lem42-cor} and \eqref{lem4.4c-claim} of \lemref{lem44-c}.
\end{rem}

\begin{proof}
Let $(\tau,Z)$ and $\htau_k$, $k\in { \NN}$, be as in the statement of the lemma.  On the set $\{\htau_{ k-1} < \infty\}$, define 
$\wdt\sigma_{ k}(\wdt\omega) := \htau_{ k}(\wdt\omega)-\htau_{ k-1}(\wdt\omega)$
to be the length of time of cycle $k-1$.  Also remember that $\htau_{ k}(\wdt\omega) = \htau_{ k}(\mbf e_{ k-1}) =\tau_{ k}(\ovl\omega_{ k-1})$, in which $\ovl\omega_0$ is set equal to $\omega_0$, and similarly $\htau_{ k-1}(\wdt\omega) = \htau_{ k-1}(\mbf{e}_{ k-2}) =\tau_{ k-1}(\ovl\omega_{ k-2})=\tau_{ k-1}(\ovl\omega_{ k-1})$.  Thus $\wdt\sigma_{ k}(\wdt\omega)  = \htau_{ k}(\wdt\omega) - \htau_{ k-1}(\wdt\omega) = \tau_{ k}(\ovl\omega_{ k-1})-\tau_{ k-1}(\ovl\omega_{ k-1}) = \sigma(\theta_{\tau_{ k-1}}(\ovl\omega_{ k-1}))$, where the last equality follows from \defref{Markov-nominal-impulse-policy}(a).  On the other hand, for any $s\ge 0$, by the definition of $\ovl\omega_{ k-1}$, 
\begin{displaymath} \vspace{-3pt} 
(\theta_{\tau_{ k-1}} (\ovl\omega_{ k-1}))(s) = \ovl\omega_{ k-1}(\tau_{ k-1} (\ovl\omega_{ k-1})+s) =\ovl\omega_{ k-1}(\htau_{ k-1}(\wdt\omega)+s) = \omega_{ k-1}(s). \end{displaymath}  
This says that $\theta_{\tau_{ k-1}}(\ovl\omega_{ k-1}) = \omega_{k-1}$ and hence $\wdt\sigma_{k}(\wdt\omega)= \sigma(\omega_{ k-1})$, showing that  $ \htau_{ k}(\wdt\omega)-\htau_{ k-1}(\wdt\omega) $  depends only on $\omega_{ k-1}$ through the function $\sigma$. This gives assertion (a). Consequently, we have $\htau_{ k}(\wdt\omega) = \sum_{i=1}^{{ k}} \sigma(\omega_{i-1})$ for all $k \in{ \NN}$. 

Now suppose $\sigma(\omega_{ k-1}) > u \ge 0$. Then we can apply \eqref{markov-stopping-time-cnd} to get
\begin{align*} 
 \sigma(\omega_{ k-1})= \sigma(\theta_{\tau_{ k-1}}(\ovl\omega_{ k-1})) &= u + \sigma(\theta_{\tau_{ k-1}+u} (\ovl\omega_{ k-1}))    \\&
   = u + \sigma(\theta_{u} (\theta _{\tau_{ k-1}} (\ovl\omega_{ k-1})) ) 
   =   u + \sigma(\theta_{u }(\omega_{ k-1})), 
\end{align*} 
completing the proof.
\end{proof} 

\begin{cor}\label{lem42-cor}
For each $j\in \NN_0$ and  $s\ge 0$, on the set  $\{\wdt\omega: \htau_{j}\le s < \htau_{j+1} \}$,  
\begin{align}\label{e:htau_j+1} 
\htau_{j+1}(\wdt\omega) = s + \sigma\circ\theta_{u}(\omega_{j}), 
\end{align} 
where $u =u(\mbf e_{j-1}) = s-\htau_{j}(\mbf e_{j-1})$, and hence for each $k\ge 2$, $\htau_{j+k} = \htau_{j+1} + \sum_{i=2}^{k} \sigma(\omega_{j+i-1}) = s + { \sigma(\theta_{u} \omega_{j})} + \sum_{i=2}^{k} \sigma(\omega_{j+i-1})$. 
\end{cor}

\begin{proof} 
The equation $\htau_{j+1}(\wdt\omega) = s + \sigma\circ\theta_{u}(\omega_{j})$ follows from \eqref{sigma-term-time} directly because $\htau_{j+1}(\wdt\omega) = \htau_{j} (\wdt\omega)+ \sigma(\omega_{j})$.
\end{proof}

Our proof of the Markov  family property for the impulse controlled process $\wdh{X}$ involves conditioning on whether the times $s$ and $t$ are in the same cycle or whether $t$ is in a later cycle.  Lemmas \ref{Fs-properties}  and \ref{lem44-c} establish results that are central to this proof. 
 \lemref{Fs-properties} is  used when $s$ and $t$ are in the same cycle whereas \lemref{lem44-c} is key to the proof when $t$ is in a later cycle.

\lemref{Fs-properties}(a) establishes the form of the $\{\htau_j \leq s < \htau_{j+1}\}$-trace $\sigma$-algebra of $\wdh\F_s$ while (b) identifies the form of sets which generate this trace $\sigma$-algebra.  Intuitively for (b), when $A=\{\wdh{X}(t_1)\in B_1,\ldots,\wdh{X}(t_n)\in B_n\}\in\wdh{\F}_s$, the generating set $\{\htau_j \leq s < \htau_{j+1}\}\cap A$ can be identified by a slice $\{\htau_j \leq s\}\cap A_{j-1} \in \wdh\F_{\htau_j-}$ in which $A_{j-1}$ is determined by those $t_i$ which are less than $\htau_j(\mbf{e}_{j-1})$ and then, conditional on $\mbf{e}_{j-1}$ being in this set, the remaining dependence is of the form $\{\sigma \geq s -\htau_j(\mbf{e}_{j-1}\}\cap \Gamma_j$ with $\Gamma_j$ corresponding to those $t_i$ which are in the $j$th cycle.  This latter set is determined by the $\sigma$-algebra generated by $\wdh{X}$ over the elapsed time in the $j$th cycle. 

\begin{lem} \label{Fs-properties}
For  a stationary  Markov nominal impulse policy  $(\tau,Z)$,  
 let  
  $\wdh{X}$ be the process  defined by \eqref{inv-proc-def}  and  $\{\wdh{\F}_t\}$  be  the natural filtration of $\wdh{X}$.  Let $\boldsymbol{\Pi}_j:\wdt\Omega \rightarrow \prod_{i=0}^j \Omega_i$ be the projection mapping onto the first $j+1$ components.  Then for each $j\in \NN_0$, 
\begin{enumerate}
\item[(a)] the $\{\htau_j \leq s < \htau_{j+1}\}$-trace $\sigma$-algebra of $\wdh{\F}_s$ consists of sets having the form $A \times \prod_{i=j+1}^\infty \Omega_i$ for some $A\in \otimes_{i=0}^j \mathscr{F}_i$ and as a result, for any $\wdh{\F}_s$-measurable random variable $W$, $I_{\{\htau_j \leq s < \htau_{j+1}\}}\cdot W$ is solely a function of $\mbf{e}_j$;
\item[(b)] the $\{\htau_j \leq s < \htau_{j+1}\}$-trace $\sigma$-algebra of $\wdh{\F}_s$ is generated by sets $\{\htau_j\leq s < \htau_{j+1}\}\cap A$ of the form
\setlength{\arraycolsep}{0.5mm}
\begin{align} \label{A-decomp}
 \{\htau_j \leq s < \htau_{j+1}\} \cap A & =    \boldsymbol{\Pi}_{j-1}^{-1}(\{\mbf{e}_{j-1}: \htau_j(\mbf{e}_{j-1}) \leq s\} \cap A_{j-1}) \\[-1pt]  \nonumber
&\quad\  \cap \, \boldsymbol{\Pi}_j^{-1}(\{\mbf{e}_j: \sigma(\omega_j) > s - \htau_j(\mbf{e}_{j-1})\} \cap \Gamma_j)
\end{align}
for some sets $A_{j-1}$ and $\Gamma_{j}$ in which $\{\htau_j \leq s\}\cap A_{j-1} \in \wdh{\cal F}_{\htau_j-}$ and $\{\mbf{e}_j: \sigma(\omega_j) > s - \htau_j(\mbf{e}_{j-1})\} \cap \Gamma_j \in \wdh{\cal F}^{\htau_j(\mbf{e}_{j-1})}_s$; 
(for $j=0$, $A_{-1}=\emptyset$; see also \eqref{gen-set-rep} for specific definitions of $A$, $A_{j-1}$ and $\Gamma_j$). 
\end{enumerate}
\end{lem}

\begin{proof} {\bf (a)} 
Recall for each $j\in \NN$, $\htau_j$ is a function of $\mbf{e}_{j-1}$ and from the previous lemma $\htau_{j} - \htau_{j-1}$ 
depends only on $\omega_{j-1}$.  Observe that the event 
\begin{eqnarray*}
\{\htau_j \leq s < \htau_{j+1}\} &=& \boldsymbol{\Pi}_j^{-1}\{\mbf{e}_j: \htau_j(\mbf{e}_{j-1}) \leq s < \htau_{j+1}(\mbf{e}_j)\} 
\end{eqnarray*}
is a set of the form $G\times \prod_{i=j+1}^\infty \Omega_i$ with $G\in \otimes_{i=0}^j \mathscr{F}_i$.    
For ease of reading in the sequel, we slightly abuse notation by dropping the reference to $\boldsymbol{\Pi}_j^{-1}$ in the right-hand side expression. 

Now consider a generating set $\{\wdh{X}(t_1)\in B_1,\ldots,\wdh{X}(t_n)\in B_n\}$ of $\wdh{\F}_s$, which means that $n\in \NN$, $0 \leq t_1 < \cdots < t_n \leq s$ and $B_i \in \B(\E)$ for $i=1,\ldots, n$.  To simplify notation, for $i=1,\ldots,n$, set $\wdh{B}_i=\{\wdh{X}(t_i)\in B_i\}$.  Since each $t_i \leq s$, on the set $\{\htau_j \leq s < \htau_{j+1}\}$, the random variable $\wdh{X}(t_i)$ is determined by a subset of the coordinate paths $\mbf{e}_j$.  As a result, $\{\htau_j \leq s < \htau_{j+1}\}\cap\wdh{B}_i \in (\otimes_{i=0}^j \mathscr F_i) \otimes (\prod_{i=j+1}^\infty \Omega_i)$.  Since the collection of $\wdh{B}_i$ generate $\wdh{\F}_s$, every set in the $\{\htau_j\leq s < \htau_{j+1}\}$-trace $\sigma$-algebra of $\wdh{\F}_s$ has the claimed form $A\times \prod_{i=j+1}^\infty \Omega_i$.  The sole dependence of the restriction of $W$ to $\{\htau_j \leq s < \htau_{j+1}\}$ on $\mbf{e}_j$ is now immediate.
\medskip

\noindent {\bf (b)}  As before, 
 we slightly abuse notation by dropping reference to $\boldsymbol\Pi_j^{-1}$.  First observe that $\{\htau_j \leq s < \htau_{j+1}\} = \{\mbf{e}_{j-1}: \htau_j(\mbf{e}_{j-1}) \leq s\} \cap \{\mbf{e}_j: \sigma(\omega_j) > s-\htau_j(\mbf{e}_{j-1})\}$.  In the analysis below, for each $\mbf{e}_{j-1}$, the index $\ell(\mbf{e}_{j-1})$ is the maximal index such that $t_{\ell(\mbf{e}_{j-1})} < \htau_j(\mbf{e}_{j-1})$.  Now using the generating sets $\wdh{B}_i$ of $\wdh{\F}_s$ with $i=1,\ldots, n$ and $n\in \NN$, we have
\begin{align} \label{gen-set-rep} 
& \{\htau_j \leq s < \htau_{j+1}\}\cap (\cap_{i=1}^n \wdh{B}_i) \\  \nonumber
&= (\{\mbf{e}_{j-1}: \htau_j(\mbf{e}_{j-1}) \leq s\}\cap (\cap_{i=1}^{\ell(\mbf{e}_{j-1})} \wdh{B}_i)) \\ \nonumber
& \; \cap (\{\mbf{e}_{j}: \sigma(\omega_j) > s-\wdh\tau_j(\mbf{e}_{j-1})\} \cap (\cap_{i=\ell(\mbf{e}_{j-1})+1}^n \wdh{B}_i)) \\  \nonumber
&= (\{\mbf{e}_{j-1}: \htau_j(\mbf{e}_{j-1}) \leq s\}\cap (\cap_{i=1}^{\ell(\mbf{e}_{j-1})} \wdh{B}_i)) \\  \nonumber
& \; \cap \{\mbf{e}_{j}: \sigma(\omega_j) > s-\wdh\tau_j(\mbf{e}_{j-1}), \wdh{X}(t_{\ell(\mbf{e}_{j-1})+1})\in B_{\ell(\mbf{e}_{j-1})+1},\ldots,\wdh{X}(t_n)\in B_n\} \\  \nonumber 
&= (\{\mbf{e}_{j-1}: \htau_j(\mbf{e}_{j-1}) \leq s\}\cap (\cap_{i=1}^{\ell(\mbf{e}_{j-1})} \wdh{B}_i)) \cap \{\mbf{e}_j: \sigma(\omega_j) > s-\wdh\tau_j(\mbf{e}_{j-1}),  \\ \nonumber
&  \qquad \qquad \omega_j(t_{\ell(\mbf{e}_{j-1})+1}-\htau_j(\mbf{e}_{j-1}))\!\in \!B_{\ell(\mbf{e}_{j-1})+1},\ldots,\omega_j(t_n-\htau_j(\mbf{e}_{j-1}))\!\in \! B_n\}.
\end{align}
In this representation, $A_{j-1} = \cap_{i=1}^{\ell(\mbf{e}_{j-1})} \wdh{B}_i$ and $\Gamma_j = \{\omega_j(t_{\ell(\mbf{e}_{j-1})+1}-\htau_j(\mbf{e}_{j-1}))\!\in \!B_{\ell(\mbf{e}_{j-1})+1},\ldots,
\omega_j(t_n-\htau_j(\mbf{e}_{j-1}))\!\in \! B_n\}$.  For those $\mbf{e}_{j-1}$ for which $t_1 > \htau_j(\mbf{e}_{j-1})$, $\ell(\mbf{e}_{j-1})+1=1$ and we set $A_{j-1}$ to be $\prod_{k=0}^{j-1} \Omega_k$ and, similarly for those $\mbf{e}_{j-1}$ such that $t_n < \htau_j(\mbf{e}_{j-1})$, $\Gamma_j = \Omega_j$.

Notice that $\{\mbf{e}_{j-1}: \htau_j(\mbf{e}_{j-1}) \leq s\} \cap A_{j-1} \in \otimes_{i=0}^{j-1} {\FFF}_i$.  Also observe that for each fixed $\mbf{e}_{j-1}$, the $\{\htau_j \leq s < \htau_{j+1}\}$-trace $\sigma$-algebra of $\wdh{\cal F}^{\htau_j(\mbf{e}_{j-1})}_s$ is the $\sigma$-algebra generated by the $j$th coordinate process $X_j$ for the elapsed time in the $j$th cycle.  Thus for fixed $\mbf{e}_{j-1}$ with $\htau_j(\mbf{e}_{j-1}) \leq s$, the second compound event $\{\mbf{e}_j: \sigma(\omega_j) > s - \htau_j(\mbf{e}_{j-1})\} \cap \Gamma_j$ in \eqref{gen-set-rep} belongs to ${\cal F}^{(j)}_{s-\htau_j(\mbf{e}_{j-1})}$.
 \end{proof}

 The next lemma characterizes the conditional probability of a future event given the information to the end of an earlier cycle.

\begin{lem}\label{lem44-c} 
Let $(\tau, Z) $ and 
 $\wdh X$ be as in the statement of \lemref{Fs-properties}.
 Let  
  $\wdh\PP^{(\htau,\wdh{Z})}_x$  be the probability measure on $(\wdt\Omega,\wdt{\cal G})$ given by the construction in \thmref{prob-construction}. 
  Then for $j\in \NN_0$ on the set $\{\htau_{j} \le s < \htau_{j+1} \le t \}$, for  any $B\in \B(\E)$, there exists a measurable function $\phi_{j}$ so that 
\begin{align} \label{lem4.4c-claim}
\nonumber \wdh\PP^{(\htau,\wdh{Z})}_{x}\{\wdh X(t) \in B, t\ge \htau_{j+1} |  \wdh{\F}_{\htau_{j+1}-} \} & = \phi_{j} (\omega_{j} (\sigma (\omega_{j})-), t-\htau_{j}(\mbf{e}_{j-1}) - \sigma(\omega_{j})) \\ 
& =  \phi_{j}([\theta_{u}\omega_j](\sigma(\theta_{u}\omega_{j})-), t-s-\sigma(\theta_{u} \omega_{j})).
\end{align} 
\end{lem}

\begin{proof} Thanks to \lemref{sigma-term-time-lemma} and \defref{Markov-nominal-impulse-policy}, for a stationary Markov nominal impulse policy $(\tau, Z) $,  for each $k\in \NN$,  on the set $\{\htau_{k} < \infty\}$, the transition function $Q$ depends only on $\wdt Y_{k}:=\omega_{k-1}(\sigma(\omega_{k-1})-)$ and $\wdt Z_{k} := \mathfrak z  (\omega_{k-1}(\sigma(\omega_{k-1})-))$.  Consequently the transition kernel $P_{k+1}$ in \eqref{trans-fn-k} only depends on the path $\omega_{k}$, not on the earlier paths $\mbf e_{k-1}=(\omega_{0},\dots, \omega_{k-1})$.

Let $A\in  \wdh{\F}_{\htau_{j+1}-}$, $j\in \NN_0$.  For $k\in \NN$ and each $\mbf{e}_{j+k-1}$, let $M_{j+k} = M_{j+k}(\mbf{e}_{j+k-1}) \\ := \{\omega_{j+k}: \omega_{j+k}(t-\htau_ {j+k}(\mbf{e}_{j+k-1})) \in B, \sigma (\omega_{j+k}) > t-\htau_ {j+k}(\mbf{e}_{j+k-1})\}$.  Then 
\begin{align*} 
  &  \wdh\PP^{(\htau,\wdh{Z})}_{x}(\{\wdh X(t) \in B \}\cap\{ t\ge \htau_{j+1} \}\cap A)  \\[-2pt]    
	& = \sum_{k=1}^{\infty}  \wdh\PP^{(\htau,\wdh{Z})}_{x}(\{\wdh X(t) \in B \}\cap A \cap \{\htau_{j+k} \le t < \htau_{j+k +1} \}) = \int I_A f_j(\mbf{e}_j)\, \wdh\PP^{(\htau,\wdh{Z})}_{x}(d\mbf{e}_j)  
	\end{align*}
in which 
\begin{align*} 
 \nonumber f_j(\mbf{e}_j) := & I_{\{ \htau_{j+1} \le t  \}}  \int_{\E}\PP_{v_{j+1}}(M_{j+1})\, Q_{(\wdt Y_{j+1}, \wdt Z_{j+1})}(dv_{j+1}) \\[-2pt] 
\nonumber &\ \  + \sum_{k=2}^{\infty}  \idotsint   I_{\{ \htau_{j+k} \le t  \}} \int_{\E}\PP_{v_{j+k}}(M_{j+k})\, Q_{(\wdt Y_{j+k}, \wdt Z_{j+k})}(dv_{j+k}) \\[-8pt] & \quad \hspace{1in}   P_{j+k-1}(\omega_{j+k-2},d\omega_{j+k-1})  \dots P_{j+1}(\omega_{j},d\omega_{j+1}). 
\end{align*} 
The dependence of $f_j$ on the transition kernel  $P_{j+1}(\omega_{j},d\omega_{j+1})$ is solely through  $\wdt Y_{j} = \omega_{j}(\sigma (\omega_{j})-)$; in addition, 
using \lemref{sigma-term-time-lemma}, the terms also depend  on $t- \htau_{j+1}(\mbf{e}_j) = t-\htau_{j}(\mbf{e}_{j-1}) - \sigma (\omega_{j})$. Thus there exists some  measurable function $\phi_{j}$ so that  $f_j(\mbf{e}_j)=\phi_{j}(\omega_{j}(\sigma(\omega_{j})-), t-\htau_{j}(\mbf{e}_{j-1}) - \sigma (\omega_{j}))$. Consequently, we have 
\begin{align*}
& \wdh\PP^{(\htau,\wdh{Z})}_{x}\{\{\wdh X(t) \in B \}\cap \{  t\ge \htau_{j+1}\}\cap A\} \\[-2pt]  & \ = \int I_{A}    \phi_{j}(\omega_{j}(\sigma(\omega_{j})-), t-\htau_{j}(\mbf{e}_{j-1}) - \sigma(\omega_{j}))   \wdh\PP^{(j)}_{x}(d\mbf e_{j})\\[-2pt] 
 & \ =  \int I_{A}   \phi_{j}(\omega_{j}(\sigma(\omega_{j})-), t-\htau_{j}(\mbf{e}_{j-1}) - \sigma(\omega_{j})) \wdh\PP^{(\htau, \wdh Z)}_{x}(d\wdt\omega).
\end{align*}  
The first equation  of \eqref{lem4.4c-claim} thus follows from the definition of conditional probability.

We next show that the second equation  of \eqref{lem4.4c-claim} holds as well. First Corollary \ref{lem42-cor}  says that on the set $\{\htau_{j} \le s < \htau_{j+1} \}$, $\htau_{j+1} = s+ \sigma(\theta_{u} \omega_{j})$ and
 hence  $ t-\htau_{j} - \sigma(\omega_{j}) = t- \htau_{j+1} 
 = t-s - \sigma(\theta_{u} \omega_{j})$. 
On the other hand, 
 on the set  $\{\htau_{j} \le s < \htau_{j+1}   \}$,  we have from  \eqref{sigma-term-time} that 
$ \sigma(\omega_{j})= u + \sigma(\theta_{u}(\omega_{j})).
$  Consequently, the dependence of $P_{j+1}(\omega_{j},d\omega_{j+1}) $ on $\omega_{j}$ is only through $ \omega_{j}( u + \sigma(\theta_{u}(\omega_{j}))-)= [\theta_{u} \omega_{j}] ( \sigma(\theta_{u}(\omega_{j}))-)$.  Using these observations in the first equation of \eqref{lem4.4c-claim} then leads to the second representation. 
\end{proof}


We now establish the Markov 
family property of $(\wdt\Omega,\wdt{\cal G},\wdh{X},\{\wdh\F_t\},\{\wdh\PP^{\htau,\wdh{Z})}_x, x\in \E\})$ corresponding to  a stationary  Markov policy $(\tau,Z)$ and then verify that the Markov family property projects onto $(\check\Omega,\check\F,X, \{\check\F_t\}, \{\PP^{(\tau,Z)}_x: x\in \E\})$. 

\begin{thm} \label{Markov-family-thm}
For   a stationary Markov nominal impulse policy $(\tau,Z)$,  $(\wdt\Omega,\wdt{\cal G},\wdh{X},\newline\{\wdh\F_t\},\{\wdh\PP^{\htau,\wdh{Z})}_x, x\in \E\})$ is a Markov family.
\end{thm}

\begin{proof}
In light of Theorem~\ref{prob-construction} 
and with reference to Definition 2.5.11 and Proposition 2.5.13 of \cite{kara:88}, the only property to prove for this family is that for $x\in \E$, $0\leq s < t$ and $B \in \B(\E)$,
\begin{equation} \label{markov-family}
 \wdh\PP^{(\htau,\wdh Z)}_x(\wdh X(t)\in B | \wdh\F_{s}) = \wdh\PP^{(\wdh\tau,\wdh Z)}_{\wdh X(s)} (\wdh X(t-s) \in B), \ \wdh\PP_{x}^{(\wdh\tau,\wdh Z)}\text{-a.s.} 
\end{equation} 
where the right-hand side equals the function $y\mapsto \wdh\PP^{(\wdh\tau,\wdh Z)}_{y} (\wdh X(t-s) \in B) $ evaluated at $y =\wdh X(s)$. Indeed, according to the proof of \thmref{prob-construction}, the function $y\mapsto \wdh\PP^{(\wdh\tau,\wdh Z)}_{y} (\wdh X(t-s) \in B) $ is universally measurable. Using the probability measure $\wdh\PP_{x}^{(\wdh\tau, \wdh Z)}   \wdh X(s)^{-1}$  on $\E$,  
  there exists a Borel measurable function $g: \E \mapsto [0,1]$ such that $g(y) = \wdh\PP^{(\wdh \tau,\wdh Z)}_{y} (\wdh X(t-s) \in B) $ for $\wdh\PP_{x}^{(\wdh\tau, \wdh Z)}   \wdh X(s)^{-1}$-a.e. $y\in \E$. This in turn  implies that   $g( \wdh X(s)) = \wdh \PP^{(\wdh\tau,\wdh Z)}_{\wdh  X(s)} ( \wdh X(t-s) \in B) $, $\wdh \PP_{x}^{(\wdh\tau,\wdh Z)}$-a.s. Therefore, if \eqref{markov-family} holds, then $\wdh\PP^{(\wdh\tau,\wdh Z)}_x(\wdh X(t)\in B |  \wdh\F_{s})$ has a $\sigma(\wdh X(s))$-measurable version. Consequently, we have  $\wdh\PP^{(\wdh \tau,\wdh Z)}_x(\wdh X(t)\in B |  \wdh \F_{s}) = \wdh\PP^{(\wdh \tau,\wdh Z)}_x(\wdh X(t)\in B | \wdh X(s)) $,  $\wdh\PP^{(\wdh\tau,\wdh Z)}_x$-a.s. The Markov family property therefore follows.  

We now establish \eqref{markov-family}. 
  It suffices   to show that for any $A\in \wdh \F_{s}$, 
\begin{equation}
\label{eq:Markov_family}
  \wdh\PP^{(\htau,\wdh{Z})}_{x}(\{   A \cap \{\wdh X_{t} \in B \}  )=    \EE^{\wdh\PP^{(\htau,\wdh{Z})}_{x}}\Big[  I_{A}  \wdh \PP^{(\htau,\wdh{Z})}_{\wdh X(s)} \{\wdh X(t-s) \in B   \}\Big].
\end{equation}
To this end, we note that thanks to the stationarity of the policy  $(\htau, \wdh Z)$, the transition kernel 
\vspace{-3pt}\begin{equation} \label{stat-trans-fn}
P_{j}(\omega_{j-1}, \cdot) = \int_{\E} \PP_{v_{j}}(\cdot) Q_{(\omega_{j-1}(\sigma(\omega_{j-1})-), \mathfrak z(\omega_{j-1}(\sigma(\omega_{j-1})-)))} (dv_{j}) \vspace{-3pt}
\end{equation}
depends on $\omega_{j-1}$ only through $\omega_{j-1}(\sigma(\omega_{j-1})-)$. Hence  we can denote $P_{j}(\omega_{j-1}, \cdot)$  by $P(\omega_{j-1}, \cdot)$ for every $j\in \NN$.  For simplicity, in the rest of the proof, we write  $Q(\omega_{j-1}, dv_{j})$ for $Q_{(\omega_{j-1}(\sigma(\omega_{j-1})-), \mathfrak z(\sigma(\omega_{j-1})-)))} (dv_{j})$ and $\wdh\PP_{x}$ for $ \wdh\PP^{(\htau,\wdh{Z})}_{x}$. In addition, let $A\in \wdh\F_{s}$ have the decomposition  \eqref{A-decomp} on the set $\{\htau_{j} \le s   < \htau_{j+1}  \}$, $j\in \NN_{0}$.  For $j\in \NN$, simplify notation by setting $F_{j-1}=A_{j-1} \cap \{\htau_j \leq s\}$ and for $\mbf{e}_{j-1}\in F_{j-1}$, set $G_j(\mbf{e}_{j-1}) = \Gamma_j \cap \{\omega_j: \sigma(\omega_j) > u(\mbf{e}_{j-1})\}\subset \Omega_j$.  Note in particular that $F_{j-1} \in \wdh{\F}_{\htau_j-}$ and $G_j(\mbf{e}_{j-1}) \in \F^{(j)}_{u(\mbf{e}_{j-1})}$. 

{\bf Case (i).} We first consider  the set  $\{\htau_{j} \le s < t < \htau_{j+1}  \}$. Recall $u=u(\mbf{e}_{j-1}): = s-\htau_{j}(\mbf{e}_{j-1})$, $v=v(\mbf{e}_{j-1}): = t-\htau_{j}(\mbf{e}_{j-1})$,  and   $M_j = M_{j}(\mbf{e}_{j-1}) : = \{\omega_{j}: \omega_{j}(v) \in B, \sigma(\omega_{j}) > v \}$ for each $\mbf{e}_{j-1}$.  
         Then
\begin{align} \label{e1-case1:family}
\nonumber &  \wdh\PP_{x} (\{ \htau_{j} \le s < t < \htau_{j+1}\} \cap A \cap \{\wdh X(t) \in B \} ) \\[-1.5pt] \nonumber  
& \ = \EE^{\wdh\PP_{x}}[I_{\{ \htau_{j} \le s\}} I_{\{\sigma(\omega_{j}) > u \}} I_{A} I_{\{\sigma(\omega_{j}) > v \}} I_{ \{\wdh X(t) \in B \}}]\\[-1.5pt]
\nonumber  & \ =  \int_{F_{j-1}} \int_{\Omega_{j}} I_{G_{j}(\mbf{e}_{j-1})}(\omega_{j})  I_{M_{j}}(\omega_j) 
P(\omega_{j-1}, d\omega_{j}) \wdh\PP^{(j-1)}_{x} (d\mbf e_{j-1})\\[-1.5pt]
\nonumber  & \ =  \int_{F_{j-1}} \int_{\E} \EE^{\PP_{v_{j}}}\big [ I_{G_{j}(\mbf{e}_{j-1})}  I_{M_{j}} 
\big] Q(\omega_{j-1}, dv_{j}) \wdh\PP^{(j-1)}_{x} (d\mbf e_{j-1})\\[-1.5pt]
  & \ =  \int_{F_{j-1}} \int_{\E} \EE^{\PP_{v_{j}}} \! \Big[ I_{G_{j}(\mbf{e}_{j-1})} \EE^{\PP_{v_{j}}}\big[ I_{M_{j}} | \F_{u}^{(j)}\big] \Big] \!Q(\omega_{j-1}, dv_{j}) \wdh\PP^{(j-1)}_{x} (d\mbf e_{j-1}).  
\end{align} 
Denote $H: = \{\omega_{j}\in \Omega_{j}: \sigma(\omega_{j}) > t-s, \omega_{j}(t-s) \in B \} $. Note that the terminal-time property for $\sigma$ implies that on the set  $\{\htau_{j} \le s  < \htau_{j+1}  \}\subset \{\sigma(\omega_{j}) > u \}$, for each $\mbf{e}_{j-1}$  
\begin{align*} 
 \theta_{u}^{-1} H   & = \{\omega_{j}\in \Omega_{j}: \sigma(\theta_{u}\omega_{j}) > t-s, (\theta_{u}\omega_{j})(t-s) \in B \} \\
  & = \{\omega_{j}\in \Omega_{j}: \sigma( \omega_{j}) > v, \omega_{j}(v) \in B \} = M_j.
\end{align*}  
Consequently   using the Markov family property of the $j$th coordinate process, we can apply property (e'') on p.~78 of  \cite{kara:88} to derive 
\vspace{-3pt}\begin{displaymath}
\EE^{\PP_{v_{j}}}\big[ I_{M_j}| \F_{u}^{(j)}\big] =\PP_{v_{j}}[ \theta_{u}^{-1} H|  \F_{u}^{(j)}]={\PP_{\omega_{j}(u)}}(H), 
    \ \ \PP_{v_{j}} \text{-a.s.}\vspace{-3pt}
\end{displaymath}
On the other hand,  for any $y\in \E$, by the construction of the measure $\wdh\PP_{y} = \wdh\PP_{y}^{(\htau, \wdh Z)}  $ in \thmref{prob-construction}, the stationarity of the policy implies 
\begin{align} \label{e2-case1:family}
\nonumber \PP_{y}(H) & =    {\PP_{y}}\{\omega_{j} \in \Omega_{j}:  \sigma(\omega_{j}) > t-s ,  \omega_{j}(t-s) \in B \} \\
\nonumber  &= {\PP_{y}} \{\omega_{0} \in \Omega_{0}:\sigma(\wdt{T}(\omega_{0})) > t-s,  \omega_{0}(t-s) \in B \} \\
 & = \wdh \PP_{y} \{\wdt\omega\in \wdt\Omega: \wdh X(t-s) \in B, \htau_{1} > t-s \}=:   \wdh \PP_{y}(\wdh H_{1}),
\end{align}  where $\wdh H_{1}: = \{\wdt\omega\in \wdt\Omega: \wdh X(t-s) \in B, \htau_{1} > t-s \} $.
 The function $y\mapsto \PP_{y}(H)$ is universally measurable, so there exists a Borel measurable function $\varphi$ such that $\varphi(y) = \PP_{y}(H)$ for  $\PP_{v_{j}} \omega_{j}(u)^{-1}$-a.e. $y\in \E$ and hence $\varphi(\omega_{j}(u)) = \PP_{\omega_{j}(u)}(H)$, $\PP_{v_{j}}$-a.s.
   Likewise, by the proof of  \thmref{prob-construction}, the function $y\mapsto  \wdh \PP_{y} (\wdh H_{1}) $
    is universally measurable and hence there exists a Borel measurable function $\psi$ so that $\psi(y) = \wdh \PP_{y} (\wdh H_{1}) $
   for  $\PP_{v_{j}} \omega_{j}(u)^{-1}$-a.e. $y\in \E$. This, in turn, implies that $\psi(\omega_{j}(u)) = \wdh \PP_{\omega_{j}(u)} (\wdh H_{1})  
     $-a.s. Then, in view of \eqref{e2-case1:family}, we have   $\varphi(y) = \psi(y)$ for  $\PP_{v_{j}} \omega_{j}(u)^{-1}$-a.e. $y\in \E$ and hence   $\varphi(\omega_{j}(u))= \psi(\omega_{j}(u))$ $\PP_{v_{j}}$-a.s. Therefore using the fact that $\omega_{j}(u) = \wdh X(s)$ on the set $\{\htau_{j} \le s   < \htau_{j+1}  \}$,  
\begin{align*}
 \EE^{\PP_{v_{j}}}\big[ I_{\{\sigma(\omega_{j}) > v \}} I_{ \{\omega_{j}(v) \in B \}}| \F_{u}^{(j)}\big] &= \PP_{\omega_{j}(u)}(H) = \wdh \PP_{\omega_{j}(u)} (\wdh H_{1}) =  \wdh \PP_{\wdh X(s)}  (\wdh H_{1}),\;  \PP_{v_{j}}\text{-a.s.}
\end{align*}  
Substituting this observation into \eqref{e1-case1:family}, we obtain   
\vspace{-2pt}\begin{align} \label{eq:case1_family}
\nonumber  &   \wdh\PP_{x} (\{ \htau_{j} \le s < t < \htau_{j+1}\} \cap A \cap \{\wdh X(t) \in B \} )  \\[-2pt]
\nonumber  & \ = \int_{F_{j-1}} \int_{\E} \EE^{\PP_{v_{j}}}\! \Big[ I_{G_{j}(\mbf{e}_{j-1})}(\omega_{j}) \wdh \PP_{\wdh X(s)}  (\wdh H_{1})\!\Big] Q(\omega_{j-1}, dv_{j}) \wdh\PP^{(j-1)}_{x} (d\mbf e_{j-1})\\[-3pt]
\nonumber  & \ =  \int_{F_{j-1}} \int_{\Omega_{j}}  I_{G_{j}(\mbf{e}_{j-1})}(\omega_{j}) \wdh \PP_{\wdh X(s)}  (\wdh H_{1})P(\omega_{j-1}, d\omega_{j}) \wdh\PP^{(j-1)}_{x} (d\mbf e_{j-1}) \\[-3pt]
  & \ =  \EE^{\wdh\PP_{x}}\Big[  I_{A}I_{\{ \htau_{j} \le s < \htau_{j+1}\}}  \wdh \PP_{\wdh X(s)} \{\wdh X(t-s) \in B, \htau_{1} > t-s \}\Big].
\end{align}

{\bf Case (ii).} We now concentrate   on the set  $\{\htau_{j} \le s  < \htau_{j+1}\le t\}$.  Due to the stationarity of the policy and the implication that each transition function is given by \eqref{stat-trans-fn}, the assertion of \lemref{lem44-c} can be strengthened as follows: for  any $B\in \B(\E)$, there exists a measurable function $\phi$ so that, for each $j$, on the set  $\{\htau_{j} \le s < \htau_{j+1} \le t \}$, 
\begin{equation}
  \label{lem4.4c-claim-new}
\begin{aligned}
 \wdh\PP_{x}\{\wdh X(t) \in B, t\ge \htau_{j+1} |  \wdh{\F}_{\htau_{j+1}-} \} = f(\mbf{e}_{j-1}, \omega_{j})= f(\mbf{e}_{j-1},\theta_{u}\omega_{j}), 
\end{aligned}
\end{equation}
where we  denote $f(\mbf{e}_{j-1}, \omega_{j}): = \phi (\omega_{j} (\sigma (\omega_{j})-), t-\htau_{j}(\mbf{e}_{j-1}) - \sigma(\omega_{j})) $ and  $f(\mbf{e}_{j-1},\theta_{u}\omega_{j}) \linebreak :=  \phi([\theta_{u}\omega_j](\sigma(\theta_{u}\omega_{j})-), t-s-\sigma(\theta_{u} \omega_{j}))$.

As in case (i), we fix an arbitrary $A\in \wdh \F_{s}$ with decomposition  \eqref{A-decomp}  and compute
\begin{align} \label{e1:case2_family}
 \nonumber  & \wdh\PP_{x}(\{ \htau_{j} \le s < \htau_{j+1}\} \cap A \cap \{\wdh   X(t) \in B \} \cap \{\htau_{j+1} \leq t\}) \\[-2pt]
 \nonumber   & \ = \EE^{\wdh\PP_{x}}\left[I_{\{ \htau_{j} \le s\}} I_{\{\sigma(\omega_{j}) > u \}} I_{A}\, \EE^{\wdh\PP_x}[I_{ \{\wdh   X(t) \in B \}} I_{\{\htau_{j+1} \leq t\}} | \wdh\F_{\htau_{j+1}-}] \right] \\[-2pt]
 \nonumber   & \ =  \int_{F_{j-1}} \int_{\Omega_{j}} I_{G_{j}}(\omega_{j}) f(\mbf{e}_{j-1},\omega_{j}) 
   P(\omega_{j-1}, d\omega_{j}) \wdh\PP^{(j-1)}_{x} (d\mbf e_{j-1})\\[-2pt]
 \nonumber   
& \ =  \int_{F_{j-1}} \int_{\E} \EE^{\PP_{v_{j}}}\big [ I_{G_{j}}(\omega_{j}) f(\mbf{e}_{j-1},\theta_{u}\omega_{j})]\, Q(\omega_{j-1}, dv_{j}) \wdh\PP^{(j-1)}_{x} (d\mbf e_{j-1})\\[-2pt]
  & \ =  \int_{F_{j-1}} \int_{\E} \EE^{\PP_{v_{j}}} \Big[ I_{G_{j}} \EE^{\PP_{v_{j}}}\big[f(\mbf{e}_{j-1},\theta_{u}\omega_{j})| \F_{u}^{(j)}\big] \Big]  
        Q(\omega_{j-1}, dv_{j}) \wdh\PP^{(j-1)}_{x} (d\mbf e_{j-1}), 
  \end{align}
 where we used  \eqref{lem4.4c-claim-new} 
 to derive the third and fourth equalities. 
 Next we apply the Markov family property (c.f.~(e'') on p.~78 of  \cite{kara:88}) to obtain 
\vspace{-2pt} \begin{equation}
\label{e2:case2_family}
\begin{aligned} 
 \EE^{\PP_{v_{j}}}\big[f(\mbf{e}_{j-1},\theta_{u}\omega_{j})
            | \F_{u}^{(j)}\big]  & = \EE^{\PP_{v_{j}}}\big[\phi([\theta_{u}\omega_j](\sigma(\theta_{u}\omega_{j})-), t-s-\sigma(\theta_{u} \omega_{j}))| \F_{u}^{(j)}\big] \\[-2pt] &  
 =  \EE^{ \PP_{\omega_{j}(u)}} [\phi(\omega_j(\sigma(\omega_{j})-), t-s-\sigma(\omega_{j}))], \ \  \PP_{v_{j}} \text{-a.s.}
\end{aligned} \vspace{-2pt}
\end{equation}
On the other hand, the stationarity of the policy $(\wdh\tau, \wdh Z)$ and the definition of the probability measure $\wdh \PP_{y}$ imply that for every $y\in \E$, we have \vspace{-2pt} \begin{align}\label{e3:case2_family}
\nonumber \EE^{ \PP_{y}} [\phi(\omega_j(\sigma(\omega_{j})-), t-s-\sigma(\omega_{j}))] & =   \EE^{ \PP_{y}} [\phi(\omega_0(\sigma(\omega_{0})-), t-s-\sigma(\omega_{0}))]\\[-2pt]
  & =    \EE^{ \wdh\PP_{y}} [\phi(\omega_0(\sigma(\omega_{0})-), t-s-\sigma(\omega_{0}))]. \vspace{-2pt}
\end{align}
Let $\wdh H_{2}: = \{ \wdh X(t-s) \in B,  t-s \ge \htau_{1}\}$. In addition, using $j=0$ in \eqref{lem4.4c-claim-new} yields 
\begin{align*} 
  & \EE^{\wdh \PP_{y}} [I_{\wdh H_{2}} | \wdh \F_{\htau_{1}-}]  =  \EE^{\wdh \PP_{y}} [I_{\{ \wdh X(t-s) \in B,  t-s \ge \htau_{1}\}} | \wdh \F_{\htau_{1}-}] =  \phi(\omega_0(\sigma(\omega_{0})-), t-s-\sigma(\omega_{0})),  \end{align*} $ \wdh\PP_{y}$-{a.s.} 
Using this representation in \eqref{e3:case2_family} gives 
\begin{equation}
\label{e4:case2_family}
 \begin{aligned} 
   \EE^{ \PP_{y}} [\phi(\omega_j(\sigma(\omega_{j})-), t-s-\sigma(\omega_{j}))] & =    \EE^{ \wdh\PP_{y}}[ \EE^{\wdh \PP_{y}} [I_{\wdh H_{2}} | \wdh \F_{\htau_{1}-}]]=  \wdh\PP_{y} (\wdh H_{2}). 
\end{aligned}
\end{equation}
 In view of \eqref{e2:case2_family} and \eqref{e4:case2_family}, detailed arguments similar to those in Case (i) reveal that 
 \begin{displaymath}
 \EE^{\PP_{v_{j}}}\big[\phi([\theta_{u}\omega_j](\sigma(\theta_{u}\omega_{j})-), t-s-\sigma(\theta_{u} \omega_{j}))| \F_{u}^{(j)}\big] = \wdh\PP_{\wdh X(s)}(\wdh H_{2}), \ \ \PP_{v_{j}}\text{-a.s.}   
\end{displaymath} 
Now we use this observation in \eqref{e1:case2_family} to obtain 
\begin{align}  \label{eq:case2_family}
\nonumber  &  \wdh\PP_{x}(\{ \htau_{j} \le s < \htau_{j+1}\} \cap A \cap \{\wdh X_{t} \in B \} \cap \{\htau_{j+1} \leq t\}) \\[-2pt]
\nonumber   & \ =   \int_{F_{j-1}} \int_{\E} \EE^{\PP_{v_{j}}} \!\Big[ I_{G_{j}}(\omega_{j}) \wdh\PP_{\wdh X(s)}  (\wdh H_{2})\!\Big]Q(\omega_{j-1}, dv_{j}) \wdh\PP^{(j-1)}_{x} (d\mbf e_{j-1})  \\[-2pt]
\nonumber  & \ =     \int_{F_{j-1}} \int_{\Omega_{j}}  I_{G_{j}}(\omega_{j}) \wdh \PP_{\wdh X(s)}  (\wdh H_{2}) P(\omega_{j-1}, d\omega_{j}) \wdh\PP^{(j-1)}_{x} (d\mbf e_{j-1})\\[-2pt]
  & \ =  \EE^{\wdh\PP_{x}}\Big[  I_{A}I_{\{ \htau_{j} \le s < \htau_{j+1}\}}  \wdh \PP_{\wdh X(s)} \{\wdh X(t-s) \in B,  t-s\ge \htau_{1}  \}\Big].
\end{align}

Finally a combination of \eqref{eq:case1_family} and \eqref{eq:case2_family} leads to 
\begin{displaymath}
 \wdh\PP_{x}(\{ \htau_{j} \le s < \htau_{j+1}\} \cap A \cap \{\wdh X_{t} \in B \}  )=    \EE^{\wdh\PP_{x}}\Big[  I_{A}I_{\{ \htau_{j} \le s < \htau_{j+1}\}}  \wdh \PP_{\wdh X(s)} \{\wdh X(t-s) \in B   \}\Big].
\end{displaymath}
Sum over $j\in \NN_{0}$ to obtain \eqref{eq:Markov_family}, completing the proof. 
\end{proof}

  Finally, we verify  that $(\check\Omega,\check\F, X,\{\check\F_t\},\{\PP^{(\tau,Z)}_x, x\in \E\})$ inherits the Markov family property from $(\wdt\Omega,\wdt{\G},\wdh{X},\{\wdh{\F}_t\},\{\wdh\PP^{(\htau,\wdh{Z})}_x, x\in \E\})$. 

\begin{thm}\label{lem-family-equivalence} 
For a stationary Markov nominal impulse policy $(\tau, Z)$, let $X$ denote the coordinate process on $\check\Omega$ and $\{\check\F_t\}$ be the natural filtration generated by $X$, including $\check\F_{0-}$.  For each $x\in \E$, let $\PP^{(\tau,Z)}_x$ be the measure on $(\check\Omega,\check\F)$ given by   Theorem \ref{prob-construction}.  Then $(\check\Omega,\check\F,X, \{\check\F_t\},\{\PP^{(\tau,Z)}_x: x\in \E\})$ is a Markov family.


\end{thm}

\begin{proof} Similar to the beginning of the proof of Theorem \ref{Markov-family-thm}, we only need to prove for any $0\le s < t$  and $B\in \B(\E)$, we have   \begin{equation}
\label{e:markov-family-Xhat} 
\PP^{(\tau,Z)}_x(X(t)\in B | \check\F_{s}) = \PP^{(\tau,Z)}_{X(s)} (X(t-s) \in B),\ \quad \PP_{x}^{(\tau,Z)}\text{-a.s.}
\end{equation}
 First we establish that $\wdh{X}^{-1}(\sigma(X(s)) = \sigma(\wdh{X}(s))$ and $\wdh{X}^{-1}( \check \F_{s} )= \wdh \F_{s}$ for any $s\ge 0$.  Since $\sigma(X(s))$ consists of sets of the form $\{X(s)\in A\}$ for $A\in {\cal B}(\E)$ and $X$ is the coordinate process on $\check\Omega$, $\{X(s) \in A\} = \{\check\omega\in \check\Omega: \check\omega(s)\in A\}$.  As a result, $\wdh{X}^{-1}(\{X(s)\in A\})= \wdh{X}^{-1}(\{\check\omega:\check\omega(s)\in A\}) = \{\wdt\omega\in \wdt\Omega: \wdh{X}(s,\wdt\omega)\in A\}.$  Similarly, by considering sets of the form $\{X(t_{1}) \in A_{1}, \dots, X(t_{n}) \in A_{n} \}$, in which $n\in \NN$ and $0\le t_{1} <  \dots <t_{n} \le s$, we obtain  $\wdh{X}^{-1}( \check \F_{s} )= \wdh \F_{s}$.
 
 By the construction of $\PP_{\cdot}^{(\tau, Z)}$ in Theorem \ref{prob-construction}, we have \begin{equation}
\label{e1:lem4.9}
\PP_{y}^{(\tau, Z)} \{X(t-s) \in B \} =\wdh \PP_{y}^{(\wdh\tau,\wdh Z)} \{\wdh X(t-s) \in B \}, \quad \forall y\in \E.
\end{equation} Let $g$ be a Borel measurable function so that $g(y) = \wdh \PP_{y}^{(\wdh\tau,\wdh Z)} \{\wdh X(t-s) \in B \}$ for $ \wdh \PP_{x}^{(\wdh\tau,\wdh Z)}\wdh X(s)^{-1}$-a.e. $y\in \E$. We have $$g(\wdh X(s)) = \wdh\PP_{\wdh X(s)}^{(\wdh \tau,\wdh  Z)} \{\wdh X(t-s) \in B \} ,  \ \ \wdh \PP_{x}^{(\wdh\tau,\wdh Z)} \text{-a.s.} $$ On the other hand, since   $ \wdh \PP_{x}^{(\wdh\tau,\wdh Z)}\wdh X(s)^{-1}= \PP_{x}^{(\tau, Z)} X(s)^{-1} $, it follows from \eqref{e1:lem4.9} that  $ 
g(y)  
=  \PP_{y}^{(\tau, Z)} \{X(t-s) \in B \}$  for $ \PP_{x}^{(\tau, Z)} X(s)^{-1}  \text{ a.e. } y\in \E$. 
 Therefore we have  
 $$g(X(s)) = \PP_{ X(s)}^{( \tau,  Z)} \{ X(t-s) \in B \}, \ \ \PP_{x}^{(\tau, Z)}\text{-a.s.}$$ Note that $g(X(s))$ is $\sigma(X(s))$ and $\check \F_{s}$-measurable. 

 For any $F\in \check \F_{s}$, $\wdh X^{-1}(F) \in \wdh\F_{s}$. 
  Then  we can use \eqref{markov-family} to derive
\begin{align*} 
 \PP_{x}^{(\tau, Z)} \{ (X(t) \in B)\cap F \} & = \wdh\PP_{x}^{(\wdh\tau,\wdh Z)} \{ (\wdh X(t) \in B)\cap \wdh X^{-1}(F) \}   \\ &  =\EE^{ \wdh\PP_{x}^{(\wdh\tau, \wdh Z)}} \big [ g(\wdh X(s))  I_{\wdh X^{-1}(F) } \big] 
  = \EE^{  \PP_{x}^{( \tau,  Z)}} \big [ g( X(s) )I_{  F } \big]. 
\end{align*} Since $F\in \check \F_{s}$ is an arbitrary set, we have $ \PP_{x}^{(\tau, Z)} \{ (X(t) \in B) | \check \F_{s}\} =  g( X(s) ) =   \PP_{ X(s)}^{( \tau,  Z)} \{ X(t-s) \in B \} $, $ \PP_{x}^{(\tau, Z)}$-a.s. This establishes \eqref{e:markov-family-Xhat}. 
\end{proof}

\section{Example}\label{sect-example}
A protypical example of an independent-cycles nominal impulse policy is an $(s,S)$ ordering policy in the inventory management of a single product, subject to some technical considerations specified below.  A more general example of such a policy in $\RR^n$ includes defining each intervention time as the hitting time of a fixed closed set $C\subset\RR^n$, subject to similar technical considerations, and then using a fixed distribution $Q$ for every intervention or different fixed distributions $\{Q_{k+1}\}$ for the different interventions.  For restricted capacity inventory control of multiple products, $C$ could be the boundary of the first orthant and $Q$ could be the Dirac measure on the maximal inventory capacity for each product.  This policy implies that an order is placed at the first time that one product's stock is depleted at which point every product is restocked to its full storage capacity.  A variant of this policy would allow each product to be fully depleted (with no back-ordering allowed) and then place an order with the goal of restoring each product to full capacity but allowing for random effects in the actual amounts delivered.

This brief section verifies that the $(s,S)$ ordering policy for a single product is both a stationary Markov policy and an independent-cycles policy.  The policy intervenes when the process hits a point and, at that instant, the process jumps to a new position specified by a random effects distribution.  Intuitively, the controlled process is continuous except at these times of intervention.  The challenge in defining the policy is that it must be defined for all paths $\check\omega\in \check\Omega$, not just those with specific discontinuites.  It is then the measure $\PP^{(\tau,Z)}_x$ which consigns all paths with discontinuities other than at the intervention times to a set of probability $0$.

Let $\E=\RR$ and fix $y\in \RR$.  The control decision is to select the target location $z > y$ to which the controlled process aims to jump.  Consider a fixed $z > y$.  The random effects distribution $Q_{(y,z)}$ is a distribution on $(y,z]$.  We assume that there is some $y_1$ with $y<y_1<z$ such that $Q_{(y,z)}([y_1,z])=1$; that is, $Q_{(y,z)}$ has its support in $(y,z]$.  To be definite, we set $y_1$ to be the maximal value for which the support of $Q_{(y,z)}$ is in $[y_1,z]$.  Define for all $\check\omega \in \check\Omega$  and $k \in \NN$,  $Z_k(\check\omega)    = z$, and the first intervention time by
\vspace{-1.3pt}
\begin{equation} \label{sS-tau1-def} \vspace{-2pt}
\tau_1(\check\omega) := \inf\{t\geq 0: \check\omega(t-)\leq y, \check\omega(s) > y \mbox{ for } 0 \leq s < t\}=: \sigma(\check\omega), \ \
\end{equation}
with $\inf\emptyset = \infty$.  For   $ \check\omega\in \check\Omega$  and $k\in \NN$, successively define the intervention times by $\tau_{k+1}(\check\omega) = \infty$ if $\check\omega(\tau_{k}) \notin [y_{1}, z]$; otherwise, put 
\begin{align*} \vspace{-1pt} 
\tau_{k+1}(\check\omega) & :=  \inf\{t \ge \tau_k(\check\omega): 
\check\omega(t-)=y, \check\omega(s) > y \mbox{ for } \tau_k(\check\omega) \leq s < t\} \\
& = \tau_{k}(\check\omega) + \sigma\circ\theta_{\tau_{k}}(\check\omega). 
\end{align*}
This policy $(\tau,Z)$ is an $(s,S)$ ordering policy with random supply for the inventory management of a single item.  It requires interventions to occur when the left continuous controlled process $\{X(t-): t\geq 0\}$ hits $y$ and precludes the process from entering the set $(-\infty,y]$ in a discontinuous manner.  In particular, observe that for $\check\omega \in \check\Omega$ with $\check\omega(0-) > y$, the initial intervention occurs at the hitting time of $y$ by $X(t-)$, which is the same rule as for all later interventions.  Thus the policy $(\tau,Z)$ satisfies the conditions of Definitions   \ref{independent-cycles-nominal-impulse-policy} and \ref{Markov-nominal-impulse-policy} as well as the additional conditions in \corref{iid-cycles}.  The facts that $\tau_k < \tau_{k+1}$ on $\{\tau_k < \infty\}$ and $\lim_{k\to\infty}\tau_{k} =\infty$ follow from the definition of $\tau_{k+1}$ and the regularity of the paths in $D_\E[0,\infty)$.  The only condition which is not immediate is that   $\tau_1$ is an $\{\check\F_{t-}\}$-stopping time and  $\tau_{k+1}$ is an $\{\F_{t-}\}$-stopping time for $k \geq 1$.  

We apply \thmref{cp65} to $\tau_1$; see \remref{cp65-tau1-mod} for the modification which covers $\tau_1$.  
 Observe that $\tau_1(\check\omega)$ is defined in terms of the behavior of the path $\check\omega$ and $\check\F = \B(\E) \otimes \F = \sigma(X(s): s\geq 0-)$ with $X$ being the coordinate process on $\check\Omega$, so $\tau_1$ is $\check\F$-measurable.  Next fix $t > 0$, let $\check\omega_1\in \{\tau_1 \leq t\}$ and suppose $\tau_1(\check\omega_1) = u \leq t$.    Let $\check\omega_2 \stackrel{R_t}{\sim} \check\omega_1$.  If $u = 0$, then $\check\omega_2(0-) = \check\omega_1(0-) \leq y$ and $\tau_1(\check\omega_2) = 0$.   Otherwise by this relation and \eqref{sS-tau1-def}, $\check\omega_2(u-) = \check\omega_1(u-) = y$ and $\check\omega_2(s) = \check\omega_1(s) > y$ for $s=0-$ and for all $0 \leq s < u$.  Again, by \eqref{sS-tau1-def}, $\tau_1(\check\omega_2) = u = \tau_1(\check\omega_1)$.  \thmref{cp65} therefore establishes that $\tau_1$ is an $\{\check\F_{t-}\}$-stopping time.  Combining induction with essentially the same argument establishes that $\tau_{k+1}$ is an $\{\F_{t-}\}$-stopping time for each $k\in \NN$.

\smallskip 

\end{document}